\documentclass[11pt,reqno,oneside]{amsart}

\usepackage[utf8]{inputenc}
\usepackage[english]{babel}
\usepackage[a4paper, total={6in, 9in}]{geometry}
\usepackage{pgf,tikz,pgfplots} \pgfplotsset{compat=newest}
\usetikzlibrary{arrows}
\usetikzlibrary{patterns}
\usetikzlibrary{arrows.meta}
\usepackage{upgreek}
\usepackage{amsmath, amscd, amssymb, amsthm, mathrsfs}
\usepackage[abbrev,nobysame,alphabetic]{amsrefs}
\usepackage{xcolor}
\usepackage{cases}
\usepackage{hyperref}
\newtheorem{thm}{Theorem}[section]
\newtheorem*{thm*}{Theorem}
\newtheorem{lem}[thm]{Lemma}

\newtheorem{prop}[thm]{Proposition}
\theoremstyle{definition}

\newtheorem{rem}{Remark}[section]
\newtheorem*{rem*}{Remark}
\numberwithin{equation}{section}

\usepackage[pagewise]{lineno}

\newcommand{\R}{\mathbb{R}}
\newcommand{\Rn}{{\mathbb{R}^n}}

\newcommand{\dif}[1]{\,\mathrm{d}{#1}}

\newcommand{\norm}[1]{\lVert #1 \rVert}

\allowdisplaybreaks

\title[Lipschitz stability simultaneous recovery]{Lipschitz stability in the simultaneous recovery of wave speed and initial source}

\author{Shiqi Ma}
\address{School of Mathematics, Jilin University, Changchun 130012, China}
\email{mashiqi@jlu.edu.cn, mashiqi01@gmail.com}

\author{Qi Zhai}
\address{School of Mathematics, Jilin University, Changchun 130012, China}
\email{zhaiqi0127@hotmail.com}

\begin{document}

\begin{abstract}
	
We consider a inverse problem for a hyperbolic equation. The objective is to simultaneous recover both the spatially varying coefficient, whose square root represents the wave speed, and the initial displacement from a single pair of Dirichlet-Neumann Cauchy data. By introducing a suitable transformation and differentiating the resulting equation in time, we derive a coupled system for the differences of the unknown coefficient and initial displacement. Combining weighted energy estimates with Carleman estimates, we establish a Lipschitz stability estimate for both unknown quantities in appropriate Sobolev norms. The proof relies on \textit{a priori} regularity assumptions, non-degeneracy conditions on the initial data, and structural inequalities controlling the interaction between the coefficient and initial-displacement differences. Our result extends coefficient-only stability results to a simultaneous coefficient--initial-displacement identification problem.

\medskip

\noindent{\bf Keywords:}~~Inverse problems, simultaneous recovery, Lipschitz stability, Carleman estimates.

\medskip

{\noindent{\bf 2020 Mathematics Subject Classification:}~~35R30, 35Q60, 35J05, 31B10, 78A40}.

\end{abstract}

\maketitle


\section{Introduction} \label{sec:intro-SS26}

Let $\Omega$ be a bounded domain in $\Rn$ with a piecewise $C^6$-smooth boundary $\partial \Omega$.
Fix $T > 0$ and denote $Q_T= (0,T) \times \Omega$, $\Sigma_T := (0,T) \times \partial \Omega$.
Let $p(x)\in C^5(\bar{\Omega})$ with $p(x) \geq C > 0$ for some constant $C$.
We investigate the following hyperbolic initial-boundary value problem:
\begin{equation*} \label{eq:1-SS26}
    \left\{\begin{aligned}
        \partial_t^2 u(t,x) - \nabla \cdot ( p(x) \nabla u(t,x)) & = 0, && (t,x) \in Q_T , \\
        u(t,x) & = g(t,x), && (t,x) \in \Sigma_T , \\
        u(0,x) = f(x), \ \partial_t u(0,x) & = 0, && x \in \Omega,
    \end{aligned}\right.
\end{equation*}

Choose an $x_0 \notin \overline{\Omega}$ satisfying
\(
c_0^2 := \min_{x \in \overline{\Omega}} |x - x_0|^2 \ge 1,
\)
and denote
\begin{equation}\label{d}
d^2=\max_{x\in\overline{\Omega}}|x-x_0|^2.
\end{equation}
Let $a$ and $M$ be two positive constants such that $a>1$ and $M>1$. Following the approach in \cite{Klibanov2006Lipschitz}, we define the class of functions $W(a,M,x_0)$ by 
\begin{equation} \label{1.2}
W(a,M,x_0) =
\left\{
\begin{aligned}
p \in C^{5}(\overline{\Omega}) :\ & a^{-1} \leq p(x) \leq 1,\ \|p\|_{C^5(\overline{\Omega})} \leq M, \\
& (x - x_0, \nabla p) \leq 0, 
\end{aligned}
\right\},
\end{equation}
where $(\cdot,\cdot)$ denotes the scalar product in $\mathbb{R}^n$. We write $h=\partial_\nu u$, where $\nu$ denotes the outward unit normal vector on $\partial\Omega$. We are now ready to state the main result of this work.

\begin{thm} \label{thm}
Let $n \in [1,13]$ be an integer. Suppose that there exist two triples $(p_1,f_1,u_1)$ and $(p_2,f_2,u_2)$ satisfying \eqref{eq:1-SS26} and
\begin{align}
p_1 |_{\partial\Omega} &= p_2 |_{\partial\Omega}, \quad
\frac{\partial^k p_1}{\partial \nu^k} \Big|_{\partial\Omega} = \frac{\partial^k p_2}{\partial \nu^k} \Big|_{\partial\Omega}, \quad k=1,2, \label{eq:con1}
\\
f_1 |_{\partial\Omega} &= f_2 |_{\partial\Omega}, \quad
\frac{\partial^k f_1}{\partial \nu^k} \Big|_{\partial\Omega} = \frac{\partial^k f_2}{\partial \nu^k} \Big|_{\partial\Omega}, \quad k=1,2,3, \label{eq:con2}
\\
u_i |_{\Sigma_T} &= g_i(t,x), \quad \frac{\partial u_i}{\partial \nu} \Big|_{\Sigma_T} = h_i(t,x), \quad i=1,2,\nonumber
\end{align}
for certain functions $g_i$, $h_i$.
Also, we assume that there exist a point $x_0 \notin \overline{\Omega}$ and numbers $a > 1$, $M>1$ such that the function $f_2$ satisfies
\begin{equation*}
	\min_{x \in \overline{\Omega}} (x - x_0, \nabla f_2(x))  > 0,
\end{equation*}
and
\begin{equation*}
	p_1,p_2 \in W(a,M,x_0).
\end{equation*}
Suppose $a<2$, and $ m\triangleq \inf_{x\in \Omega} |\nabla p_1|>0$. Denote $A=\Delta(\nabla (p_1-p_2) \cdot \nabla f_2)$ and $C=p_1\Delta (p_1^2\Delta (f_1-f_2))$. We assume that
\begin{align} 
    |A|\leq \frac{1}{8a^8} |C|,\quad \text{in } \Omega \label{1.901}
\end{align}
Denoting
\[T_0=\max\{\frac{4c_0(1+M)}{2d^{1/4}m^{1/2}},\frac{32c_0}{m^{3/2}d^{3/4}},\frac{4M^{1/2}d^{1/4}(d^2-c_0^2)}{c_0}\}\]
assume $T_0<T$
and $u_i \in C^6(\overline{Q_T})$, 
\begin{equation*}
	\|u_i\|_{C^6(\overline{Q_T})} \leq M_1,
\end{equation*}
where $M_1$ is a positive constant. Then there exists a positive constant $N$ depending on $a$, $M$, $M_1$, $\Omega$, $x_0$, $T$, $\norm{f_1}_{C^4(\overline\Omega)}$, $\norm{f_2}_{C^4(\overline\Omega)}$, $\min_{x\in\overline\Omega}(x-x_0,\nabla f_2)$ and $\inf_{x\in\Omega}|\nabla p_1|$ such that the following Lipschitz stability estimate holds:
\begin{equation*}
    \begin{aligned}
        \|p_1 - p_2\|^2_{H^2(\Omega)} + \|f_1 -f_2\|^2_{H^1(\Omega)} \leq N (\left\| \partial_t^4(g_1 - g_2) \right\|^2_{H^{1}(\Sigma_T)} + \left\| \partial_t^4(h_1-h_2) \right\|^2_{L^{2}(\Sigma_T)}).
    \end{aligned}
\end{equation*}
\end{thm}

\begin{rem}
Condition \eqref{1.901} is a structural assumption on the unknown differences $p_1-p_2$ and $f_1-f_2$ (rather than on the observed data) that breaks the symmetry between the two unknowns. Indeed, by \eqref{mathcalD} the initial datum $\bar v_t(0,\cdot)=\mathcal D(p_1,p_2,f_1,f_2)$ decomposes as
\[
\mathcal D = p_1\Delta(p_1^2\Delta(f_1-f_2)) + p_1p_2\Delta(\nabla(p_1-p_2)\cdot\nabla f_2) + (\text{lower order}) =: C + p_1p_2 A + (\text{lower order}),
\]
where $A$ and $C$ are as in \eqref{1.901}. From the elementary inequality $(X+Y)^2\ge\frac12 X^2-Y^2$ one may isolate either $C$ or $p_1p_2A$, but not both at once: the two lower bounds $\frac12 C^2-(p_1p_2A)^2$ and $\frac12(p_1p_2A)^2-C^2$ cannot be simultaneously positive unless one of the two leading terms is small relative to the other. Condition \eqref{1.901} requires precisely that the $q$-information be pointwise dominated by the $f$-information, which allows one to recover $f_1-f_2$ first and then $p_1-p_2$ through the wave equation. The pointwise condition may be relaxed to the weighted integral form
\[
\int_{\Omega}\bigl|\Delta(\nabla(p_1-p_2)\cdot\nabla f_2)\bigr|^2 e^{2\lambda|x-x_0|^2}\dif x \le \frac{1}{16}\int_{\Omega}\bigl|p_1\Delta(p_1^2\Delta(f_1-f_2))\bigr|^2 e^{2\lambda|x-x_0|^2}\dif x,
\]
which already suffices for the absorption carried out in the proof of Lemma \ref{lemma-3.4}. A typical admissible configuration is obtained by taking $f_1-f_2$ sufficiently oscillatory so that $\Delta(p_1^2\Delta(f_1-f_2))$ is large, while $p_1-p_2$ remains comparatively smooth.
\end{rem}

To guarantee that functions $u_1,u_2\in C^6(\overline{Q_T})$, one needs to impose some compatibility and smoothness conditions on the functions $f_1$, $f_2$ and coefficients $p_1$, $p_2$, see, e.g., \cite{Ladyzhenskaya1985}*{Chapter 4}.
We are not formulating such conditions in this article, since we are focused exclusively on the inverse problem. To simplify the presentation, we are not concerned here with weakening smoothness conditions.

Our work uses Carleman estimates and weighted energy estimates to prove the result.
Carleman estimates were introduced in the field of inverse problems for the first time by Bukhgeim and Klibanov in \cite{BukhgeimKlibanov}.
In that paper, the authors used Carleman estimates to give stability, and this was summarized in \cite{Klibanov1992}.
Following the scheme of the Bukhgeim-Klibanov method, many studies have been conducted.
\cite{IY01} studied the homogeneous Neumann boundary data case, i.e., $\partial_\nu u = 0$.
\cite{Imanuvilov2001} proved a Lipschitz stable recovery for the potential using interior data $u_1 - u_2$ in $(0,T) \times \omega$ where $\omega \subset \Omega$ satisfies
\(
\{ x \in \partial \Omega \,;\, (x - x_0) \cdot \nu(x) \geq 0\} \subset \partial \omega
\)
for a fixed $x_0 \notin \overline \Omega$.
In \cites{IY01, Imanuvilov2001}, the dimension is limited to $n \leq 3$.
\cites{Be04, Klibanov2006Lipschitz} applied the B-K method to the sound speed case and obtained Lipschitz stability results.
Klibanov in his recent monograph
\cite{Klibanov2021book}*{Section 3.6} obtained a Lipschitz stability for the sound speed.

The recent works \cites{Klibanov2022point, Klibanov2023plane} applied this method to inverse scattering problems for the sound speed in which the hyperbolic system is driven by a point source.
They obtained H\"older stability estimates.
The B-K method has also been applied to other types of PDEs.
For example, it was used for parabolic equations in \cite{IY1998parabolic}, and for Schr\"odinger equations in \cites{Baudouin2002Schrodinger, Huang2020Stability}.
In the case of partial data, Isakov and Yamamoto \cite{Isakov2003Stability} established Lipschitz stability for an inverse source problem using measurements on a portion $(0,T) \times \Gamma$ of the lateral boundary, where $\Gamma \subset \partial \Omega$.
Imanuvilov and Yamamoto \cite{IY2003Determination} established H\"older stability for recovering the sound-speed coefficient from a single interior observation on $(0,T)\times\omega$, where $\omega\subset\Omega$ satisfies suitable geometric conditions.
Bellassoued proved logarithmic stability for the potential in \cite{Bellassoued2004log}, using the lateral boundary data on $(0,T) \times \Gamma$ where $\Gamma$ is an arbitrarily small subset of $\partial \Omega$.
Then, this logarithmic stability was generalized to the sound speed case in \cite{Bellassoued2006acoustic} and to the Schr\"odinger equation in \cite{Bellassoued2009Schrodinger}.
There are other works in this field and we only listed a few.
For more references, we refer the readers to monographs \cites{Klibanov2012book, Bellassoued2017book, Klibanov2021book, Yamamoto2025book} and references therein.

Our work is closely related to the pioneering work of Imanuvilov and Yamamoto \cite{IY2003Determination}, where a H\"older-type stability estimate for recovering the unknown coefficient \(p\) from a single interior observation was established by means of Carleman estimates. Subsequent developments have considered the simultaneous identification of coefficients and sources from highly limited measurements. In particular, Kian and Uhlmann \cite{Kian2025} established unique simultaneous recovery of the sound speed and an initial source from a single passive boundary measurement in photoacoustic tomography.
Chen et al.~\cite{Chen2025} developed a unified framework for simultaneous identification of coupled unknowns in evolutionary equations from a single passive boundary observation, while Kian and Liu \cite{kian:hal-05158836} further established uniqueness and H\"older stability for the simultaneous recovery of the wave speed and an initial source in the wave equation.
 
The present work develops a simultaneous stability analysis for both the coefficient \(p\) and the initial displacement \(f\) within the Carleman estimate framework. Under the stated \textit{a priori} assumptions on the data and coefficients, we derive a conditional Lipschitz stability estimate that quantitatively controls the errors in both unknown quantities.
Thus, in the setting considered here, our result treats a coupled coefficient--initial-displacement identification problem using lateral boundary Cauchy data. It is complementary to the coefficient-only framework of \cite{IY2003Determination} and to recent studies of simultaneous recovery from single passive boundary measurements \cites{Kian2025, Chen2025, kian:hal-05158836}, which use different observation settings and admissible classes.

The remainder of this paper is organized as follows.
Section \ref{sec:prelim-SS26}  introduces the transformed equations and recalls the weighted energy and Carleman estimates used throughout the proof.
Section \ref{sec:lemmas-SS26} establishes several auxiliary estimates, including the weighted integral estimate and the energy estimate for the hyperbolic equation.
Section \ref{sec:proof-SS26} combines these estimates to prove the Lipschitz stability result stated in Theorem \ref{thm}.

\section{Preliminaries} \label{sec:prelim-SS26}

For convenience, we introduce the following notation:
\begin{equation} \label{eq:calL}
	\mathcal L_p := p \Delta - \nabla p \cdot \nabla - \Delta p, \quad
	F[p] := p^{-1} |\nabla p|^2.
\end{equation}
Here, $\mathcal L_p$ is a second-order differential operator which depends on the function $p$, and $\mathcal L_p$ is linear with respect to $p$ in the sense that $\mathcal L_{\alpha_1 p_1 + \alpha_2 p_2} = \alpha_1 \mathcal L_{p_1} + \alpha_2 \mathcal L_{p_2}$ for constants $\alpha_1$, $\alpha_2$.
Meanwhile, $F\colon p\mapsto F[p]$ is a nonlinear functional.
 
Using these notations, we directly derive the equations for \,$U(t,x)=p(x)u(t,x)$\,, 
\begin{equation*}
    \left\{\begin{aligned}
    \partial_t^2 U &= (\mathcal L_p + F[p]) U, && (t,x) \in Q_T,\\
    U&=p(x)g(t,x),&&(t,x)\in \Sigma_T, \\
    U(0,x)&=p(x)f(x), \ \ U_t(0,x)=0,&&x\in \Omega. \\
    \end{aligned}\right.
\end{equation*}
Then taking the time derivative gives 
\begin{equation} \label{U3}
    \left\{\begin{aligned}
        \partial_t^2 (\partial_t^3 U) &= (\mathcal L_p + F[p]) \partial_t^3 U, && (t,x) \in Q_T,\\
        \partial_t^3 U &= p(x) \partial_t^3 g(t,x),&&(t,x) \in \Sigma_T\\
    	\partial_t^3 U(0,x) &= 0,&&x\in \Omega,\\
    	\partial_t (\partial_t^3 U)(0,x) &= (\mathcal L_p + F[p])^2 (fp), &&x\in \Omega.
    \end{aligned}\right.
\end{equation}

To investigate the inverse problem, we consider two sets of coefficients $p_1,f_1$ and $p_2,f_2$ and the corresponding solutions $u_1,u_2$. 
Let $U_j = p_j u_j~(j=1,2)$, denote 
\begin{equation*}
	v = \partial_t^{3} (U_1 - U_2), \quad q = p_1 - p_2,\quad \hat{f}=f_1-f_2
\end{equation*}
\begin{equation*}
    \widetilde{g}=p_1(x)\partial_t^{3}(g_1-g_2), \quad  \widetilde{h}=p_1(x)\partial_t^{3}(h_1-h_2)+\frac{\partial p_1}{\partial \nu}\partial_t^{3}(g_1-g_2). 
\end{equation*}
Then we have the corresponding initial boundary value problem for $v$:
\begin{equation} \label{1.4}
    \left\{\begin{aligned}
        \partial_{t}^2 v=&(\mathcal L_{p_1} + F[p_1])v + \mathcal{A}(p_1,p_2), && (t,x) \in Q_T,\\
        v=& \tilde{g},&&(t,x)\in \Sigma_T,\\
        v(0,x)=&0,&&x\in \Omega,\\
        v_t(0,x)=&(\mathcal L_{p_1} + F[p_1])^2 (f_1p_1)-(\mathcal L_{p_2} + F[p_2])^2 (f_2p_2),&&x\in \Omega,
    \end{aligned}\right.
\end{equation}
where
\begin{equation*} \label{eq:calA}
	\mathcal A (p_1,p_2) := (\mathcal L_{p_1 - p_2} + F[p_1] - F[p_2]) (\partial_t^3 U_2).
\end{equation*}

Define $w(t,x) := \partial_t v(t,x)$. Then we have the corresponding initial boundary value problem for $w$:
\begin{equation} \label{w}
    \left\{\begin{aligned}
        \partial_t^2 w=&(\mathcal L_{p_1} + F[p_1])w + \partial_t\mathcal{A}(p_1,p_2),  && (t,x) \in Q_T,\\
        w=&\tilde{g}_t,\frac{\partial w}{\partial \nu}=\tilde{h}_t,&&(t,x)\in \Sigma_T,\\
        w(0,x)=&(\mathcal L_{p_1} + F[p_1])^2 (f_1p_1)-(\mathcal L_{p_2} + F[p_2])^2 (f_2p_2),&&x\in \Omega,\\
        w_t(0,x)=& \mathcal{A}(p_1,p_2), &&x\in \Omega.
    \end{aligned}\right.
\end{equation}
Since $\partial_t^3U_2(0,x)=0$ by \eqref{U3}, we have $\mathcal A(p_1,p_2)(0,x)=0$, and hence
\begin{equation}\label{eq:wt0}
    w_t(0,x)=0,\quad x\in \Omega.
\end{equation}
By \eqref{1.4}
\[ v(0,x)|_{\partial \Omega}=\frac{\partial v(0,x)}{\partial \nu} \Big|_{\partial \Omega}=0. \]
Hence
\[\tilde{g}(t,x)=\int_{0}^{t}\tilde{g}_\tau(\tau,x)  \,d\tau ,\quad  \tilde{h}(t,x)=\int_{0}^{t}\tilde{h}_\tau(\tau,x)  \,d\tau, \] 
which implies that 
\begin{equation}\label{g}
    \norm{\widetilde{g}}^2_{H^{1}(\Sigma_T)}\leq N\norm{\widetilde{g}_t}^2_{H^1(\Sigma_T)},\quad
    \norm{\widetilde{h}}^2_{L^{2}(\Sigma_T)}\leq N\norm{\widetilde{h}_t}^2_{L^2(\Sigma_T)}.
\end{equation}\par 
Consider the Carleman weight function 
\begin{equation}\label{CWF}
    \mathcal{C}(t,x)=e^{\lambda \psi(t,x)}, \quad \psi(t,x)=|x-x_0|^2-\beta t^2,
\end{equation}
where $\lambda$ is a positive constant and
\[
\beta := T^{-1} \sqrt{\frac{c_0^2}{4Md^{1/2}}}.
\]
The following lemma is a reformulation of Theorem 2.2.4 in \cite{Klibanov2012book}, which is the source of $n \in[1,13]$.
\begin{lem}\label{lemma-1}
Suppose $\Omega$ is a bounded domain in $\Rn$, $n\in [1,13]$ is an integer, and $r \in C^1(\overline{\Omega})$ satisfies
\[1\leq r(x)\leq a,\quad \frac{1}{2}+(x-x_0,\nabla r(x))>0, \quad x\in  \Omega.\]
Let $\delta = \delta(\Omega, x_0, T)$ be a sufficiently small positive constant. Suppose that the number $c^2 \in \bigl(c_0^2 - \delta, c_0^2 - \delta/2\bigr)$. Then there exist a sufficiently large positive constant $\lambda_0 = \lambda_0(\Omega, x_0, \beta) > 1$ and a positive number $C = C(\Omega, x_0, \beta)$ such that the following pointwise Carleman estimate is valid for all functions $u \in C^2(\overline{Q}_T)$:
\[
\bigl(r(x)u_{tt} - \Delta u\bigr)^2 \mathcal{C}^2 \geq C\lambda\bigl(|\nabla_{t,x} u|^2 + \lambda^2 u^2\bigr)\mathcal{C}^2 + \nabla \cdot \mathcal{U} + \partial_t V \quad \text{in } G_{c^2},
\]
where the $\R^{n+1}$-valued vector function $(\mathcal{U}, V)$ satisfies the following estimate:
\[
|(\mathcal{U}, V)| \leq C\lambda\bigl(|\nabla_{t,x} u|^2 + \lambda^2 u^2\bigr) \mathcal{C}^2 \quad \text{in } G_{c^2}.
\]

\end{lem}

Lemma \ref{lemma-2} gives an estimate for the integral involving the Carleman weight function $\mathcal{C}$, which was previously proved in Lemma 3.1.1 of \cite{Klibanov2012book}*{Chapter 3}.
\begin{lem}\label{lemma-2}
    For all functions $s \in C(\overline{G_{c^2}})$ and for all $\lambda \geq 1$, the following estimate holds:
\[
\int_{G_{c^2}} \left[ \int_{0}^{t} s(\tau, x) \dif \tau \right]^2 \mathcal{C}^2(t, x) \dif x\dif t \leq \frac{1}{\lambda \beta} \int_{G_{c^2}} (s^2 \mathcal{C}^2)(t, x) \dif x\dif t.
\]
\end{lem}

Lemma \ref{lemma-3} is an analogue of the classical energy estimate for hyperbolic operators of the second order, see \cite{Ladyzhenskaya1985}*{Chapter 4}.
\begin{lem}\label{lemma-3}
    Let $s\in C(\overline{\Omega})$, and let $s_0,s_1$ be positive constants such that $s_0\leq s(x) \leq s_1$ for all $x \in \Omega$.
    Suppose also that $u\in C^2(\overline{Q}_T)$ satisfies the hyperbolic inequality
\begin{equation*}
\bigl|s(x)u_{tt} - \Delta u\bigr| \leq B\bigl(|\nabla_{t,x} u| + |u| + |z(t, x)|\bigr), \quad \text{in } Q_T,
\end{equation*}
as well as initial and boundary conditions
\[
u(x, 0) \in H^1(\Omega), \quad u_t(x, 0) \in L^2(\Omega),
\]
\[
u\big|_{\Sigma_T} \in H^1(\Sigma_T), \quad \frac{\partial u}{\partial \nu} \Big|_{\Sigma_T} \in L^2(\Sigma_T),
\]
where $B$ is a positive constant. Then there exists a positive constant $B_1 = B_1(\Omega, T, B, s_0, s_1)$ depending on $\Omega, T, B, s_0$ and $s_1$ such that
\begin{equation*}
\|u\|_{H^1(Q_T)} \leq B_1\bigl(\|u(0,\cdot)\|_{H^1(\Omega)} + \|u_t(0,\cdot)\|_{L^2(\Omega)} + \|u\|_{H^1(\Sigma_T)} + \|\partial_\nu u\|_{L^2(\Sigma_T)} + \|z\|_{L^2(Q_T)}\bigr).
\end{equation*}
\end{lem}

We formulate Lemma \ref{lemma-5}, which was proved in \cite{Klibanov2012book}. In this lemma, we state the following standard Carleman estimate for the Laplacian, which we will use to convert the higher-order terms of $f_1-f_2$ into lower-order ones.

\begin{lem}\label{lemma-5}
	Let $n\in [1,13]$ be an integer, and let $\mathcal{C}(x)=e^{\lambda|x-x_0|^2}$ be the spatial Carleman weight and $\lambda>0$ be a positive parameter. Then there exist a sufficiently large positive constant $\lambda_0 = \lambda_0(\Omega, x_0) > 1$ and a positive number $C = C(\Omega, x_0)$ such that the following pointwise Carleman estimate is valid for all functions $u \in C^2(\overline{\Omega})$:
	\begin{align}
		(\Delta u)^2 \mathcal{C}^2 \geq C\lambda(|\nabla u|^2 + \lambda^2 u^2) \mathcal{C}^2 + \nabla \cdot W,\quad x\in \Omega.\nonumber
	\end{align}
	where the function $W$ satisfies the inequality 
	\[|W|\le C\lambda (|\nabla u|^2+\lambda^2u^2)\mathcal{C}^2\]
\end{lem}
For a time-independent $u$, integrating in time over $(0,\xi)$ and using $e^{-2\lambda\beta t^2}\le1$ yields the same estimate with $e^{2\lambda|x-x_0|^2}$ replaced by the space-time weight $\mathcal C^2(t,x)$ of \eqref{CWF} on $Q_\xi$.

Next, we present a Carleman estimate for a third order operator, which was proved in \cite{Klibanov2006Lipschitz}.

\begin{lem} \label{lemma-4}
	Let the function $f \in C^3(\overline{\Omega})$ satisfy $\min_{x \in \overline{\Omega}} (x - x_0, \nabla f(x))\triangleq \mu (x_0) > 0$. Then there exist a positive constant $C = C(\Omega, x_0, \|f\|_{C^3(\overline{\Omega})}, \mu(x_0))$, a Carleman weight function $\mathcal{C}=e^{\lambda |x-x_0|^2}$ and a sufficiently large positive constant $\lambda_0 = \lambda_0(\Omega, x_0, \mu(x_0)) > 1$ such that for all functions $u \in C^3(\overline{\Omega})$ the following pointwise Carleman estimate is valid:
	\begin{equation*}
		[\Delta(\nabla u \cdot \nabla f)]^2 \mathcal{C}^2 \geq C\lambda \sum_{i,j=1}^n u_{ij}^2 \mathcal{C}^2 + C\lambda^3(|\nabla u|^2 + \lambda^2 u^2) \mathcal{C}^2 + \nabla \cdot W,
	\end{equation*}
	where the vector function $W$ satisfies
	\begin{equation*}
    	|W|\leq C\lambda \sum_{i,j=1}^{n} u_{ij}^2 \mathcal{C}^2+ C\lambda^3 \left( |\nabla u|^2 + \lambda^2 u^2 \right) \mathcal{C}^2.
	\end{equation*}
	In particular,
	\begin{equation*}
    	\int_{\Omega}\left[\Delta(\nabla u\cdot\nabla f)\right]^2\mathcal{C}^2\geq C\lambda \sum_{i,j=1}^{n}\int_{\Omega} u_{ij}^2 \mathcal{C}^2 \, dx+C\lambda^3\int_{\Omega} \left(|\nabla u|^2+\lambda^2u^2\right)\mathcal{C}^2 \, dx.
	\end{equation*}
	for all real valued functions $u\in H^3(\Omega)$ satisfying
	\[
    u\big|_{\partial\Omega} = \frac{\partial^k u}{\partial\nu^k} \Big |_{\partial\Omega}=0, \qquad k=1,2.
	\]
\end{lem}

We also need another pointwise Carleman estimate that is important for our proof, which was established in \cite{ma2025init}*{Proposition 2.3}.

\begin{prop}[Pointwise Carleman estimate] \label{prop-2511}
	Let $a^{jk}(x)$ be a real symmetric matrix-valued, $C^2$-smooth function.
	Let $u$ and $\tilde \ell$ be real-valued and $C^2$-smooth, and $\ell$ be real-valued and $C^3$-smooth.
	Then for $v = e^\ell u$ there holds,
	\begin{align*}
		e^{2\ell} |a^{jk} \partial_{jk} u|^2
		& \geq 2 \big[ 2 (a^{jk'} a^{j'k} \ell_{j'})_{k'} - (a^{jk} a^{j'k'})_{k'} \ell_{j'} - \tilde \ell a^{jk} \big] v_j v_k + 2 (a^{jk} a^{j'k'} \ell_{j'k'} - a^{jk} \tilde \ell)_j v v_k \nonumber \\
		& \quad + \big\{ 2 a^{jk} a^{j'k'} (\ell_{j'} \ell_{k'})_j \ell_k + 2(a^{jk} a^{j'k'})_j \ell_{j'} \ell_{k'} \ell_k + 2a^{jk} \ell_j \ell_k \tilde \ell \big\} v^2 \nonumber \\
		& \quad - \big\{ 2 a^{jk} a^{j'k'} (\ell_k \ell_{j'k'})_j + 2(a^{jk} a^{j'k'})_j \ell_{j'k'} \ell_k + 2\tilde \ell^2 - 2a^{jk} \ell_{jk} \tilde \ell \big\} v^2 + \partial_j F^j,
	\end{align*}
	where
	\begin{equation*}
		F^j := a^{jk} a^{j'k'} (2 \ell_k v_{j'} v_{k'} - 4 \ell_{j'} v_{k} v_{k'} - 2\ell_{j'k'} v_k v - 2 \ell_{j'} \ell_{k'} \ell_k v^2 + 2\ell_k \ell_{j'k'} v^2) + 2\tilde \ell a^{jk} v_k v.
	\end{equation*}
\end{prop}

Lemma \ref{2511-2.5} was proved in \cite{ma2025init}*{Lemma 2.5}.

\begin{lem} \label{2511-2.5}
	Fix $x_0 \in \mathbb{R}^n \backslash \overline{\Omega}$ such that $\operatorname{dist}(x_0, \Omega) \geq 1$.
	There exist two positive constants $\lambda_0$ and $C = C(\lambda_0)$ such that for all $\lambda \geq \lambda_0$ and all $u \in H^2(\Omega)$, there holds
	\begin{equation*}
		C \int_{\Omega} e^{2\lambda |x - x_0|^2} ( \lambda |\nabla u|^2 + \lambda^3 u^2 )
		\leq
		\int_{\Omega} e^{2\lambda |x - x_0|^2} |\Delta u|^2 + \int_{\partial \Omega} e^{2\lambda |x - x_0|^2} ( \lambda |\nabla u|^2 + \lambda^3 u^2 ).
	\end{equation*}
\end{lem}

\section{Some related lemmas} \label{sec:lemmas-SS26}

Choose a constant $c$ such that $c^2 \in (0, c_0^2)$ and consider the domain $G_{c^2}$ defined by
\[
G_{c^2}=\{(t,x)\in Q_T: \psi(t,x)>c^2\}.
\]

\begin{figure}[htbp]
\centering

\begin{tikzpicture}[
    >=Latex,
    font=\small,
    line cap=round,
    line join=round
]


\draw[->,thick]
(-1.65,0)--(6.75,0)
node[below right]{$x$};

\draw[->,thick]
(0,0)--(0,4.35)
node[above]{$t$};


\draw[thick] (0,0)--(0,4);
\draw[thick] (6,0)--(6,4);


\draw[very thick,green!60!black]
(0,0)--(6,0);

\node[below=5pt] at (3,0) {$\Omega$};


\fill (-1.2,0) circle (1.7pt);
\node[below left] at (-1.2,0) {$x_0$};

\draw[<->,thick]
(-1.2,-0.42)--(0,-0.42);

\node[below] at (-0.6,-0.42)
{$\operatorname{dist}(x_0,\Omega)\ge1$};


\fill[purple!15]
(0,0)
--
plot[
domain=0:6,
samples=180
]
(\x,{sqrt(((\x+1.2)^2-1)/2)})
--
(6,0)
--cycle;


\draw[very thick,blue]
(0,{sqrt(((0+1.2)^2-1)/2)})
plot[
domain=0:6,
samples=180
]
(\x,{sqrt(((\x+1.2)^2-1)/2)});


\draw[very thick,red]
(0,0)--
(0,{sqrt(((0+1.2)^2-1)/2)});

\draw[very thick,red]
(6,0)--
(6,{sqrt(((6+1.2)^2-1)/2)});


\node at (3.15,1.05)
{$G_{c^2}$};


\draw[->,thick,blue]
(4.55,3.30)
--
(4.15,{sqrt(((4.15+1.2)^2-1)/2)+0.06});

\node[
anchor=west,
text=blue
]
at(4.60,3.35)
{$\partial_1G_{c^2}$};


\draw[->,thick,red]
(-0.55,0.22)
--
(-0.02,0.22);

\node[
anchor=east,
text=red
]
at(-0.58,0.22)
{$\partial_2G_{c^2}$};


\draw[->,thick,red]
(6.55,1.55)
--
(6.02,1.55);

\node[
anchor=west,
text=red
]
at(6.58,1.55)
{$\partial_2G_{c^2}$};


\draw[->,thick,green!60!black]
(2,-0.95)
--
(2,-0.06);

\node[
text=green!60!black,
align=center
]
at(2,-1.22)
{$\partial_3G_{c^2}$};

\end{tikzpicture}

\caption{
Schematic one-dimensional cross-section illustration of $G_{c^2}$ and the decomposition
\[
\partial G_{c^2}
=
\partial_1G_{c^2}
\cup
\partial_2G_{c^2}
\cup
\partial_3G_{c^2},
\]
where $\partial_1G_{c^2}$ is the curved boundary,
$\partial_2G_{c^2}$ consists of the two lateral boundary segments,
and $\partial_3G_{c^2}$ is the bottom boundary.
}

\end{figure}

The boundary of $G_{c^2}$ can be divided into three parts:
\begin{equation*}
    \partial G_{c^2} := \partial_1 G_{c^2}\cup \partial_2 G_{c^2} \cup \partial_3 G_{c^2},
\end{equation*}
where
\begin{equation*}
    \left\{\begin{aligned}
        \partial_{1}G_{c^2}&=\{(t,x)\in Q_T: \psi(t,x)=c^2\},\\
        \partial_{2}G_{c^2}&=\{(t,x)\in \Sigma_T: \psi(t,x)>c^2\},\\
        \partial_{3}G_{c^2}&=\{(t,x): x\in \Omega ,t=0\}.
    \end{aligned}\right.
\end{equation*}
Recall that $\beta=\frac{c_0}{2M^{1/2}d^{1/4}}T^{-1}$ by \eqref{CWF}, so that $\beta T^2=\frac{c_0T}{2M^{1/2}d^{1/4}}$. In view of the third term in the definition of $T_0$ in Theorem \ref{thm}, the assumption $T_0<T$ implies $\beta T^2>2(d^2-c_0^2)$. Choose $\delta$ with $0<\delta<\min\{c_0^2,\ d^2-c_0^2\}$ and $c^2\in(c_0^2-\delta, c_0^2-\frac12\delta)$. Then
\[
\beta T^2>2(d^2-c_0^2)>d^2-c_0^2+\delta>d^2-c^2,
\]
so that $\psi(T,x)=|x-x_0|^2-\beta T^2\le d^2-\beta T^2<c^2$ for all $x\in\Omega$, whence
\[
\overline{G_{c^2}}\cap\{t=T\}=\emptyset.
\]
Furthermore, choose a sufficiently small positive constant $\sigma$ such that
\begin{equation}\label{sigma}
c^2+3\sigma\in (c_0^2-\delta ,c_0^2-\frac{1}{2}\delta ).
\end{equation}
Obviously, $G_{c^2+3\sigma}\subset G_{c^2+2\sigma}\subset G_{c^2+\sigma}\subset G_{c^2}$. Finally, since $c^2+3\sigma<c_0^2$, choose $\xi$ with
\[
\xi=\frac12\sqrt{\frac{c_0^2-(c^2+3\sigma)}{\beta}},
\]
so that $\psi(t,x)\ge c_0^2-\beta t^2>c^2+3\sigma$ for $t<\xi$ and $x\in\Omega$, which gives
\begin{equation*}
    Q_{\xi }:= (0,\xi )\times\Omega\subset G_{c^2+3\sigma}.
\end{equation*}
By using Condition \eqref{d}, we can know that
\begin{equation}\label{3.4}
    \mathcal{C}^2(t, x)
    \begin{cases}
        \geq \exp\left[2\lambda(c^2 + 3\sigma)\right], & \text{if } (t, x) \in G_{c^2 + 3\sigma}, \\
        \leq \exp\left[2\lambda(c^2 + 2\sigma)\right], & \text{if } (t, x) \in G_{c^2} \backslash G_{c^2 + 2\sigma},\\
        \leq \exp(2\lambda d^2),& \text{if } (t,x)\in \overline{Q}_T.
    \end{cases}
\end{equation}

\begin{lem} \label{lemma-3.1}
    For all functions $s=s(x)\in C(\overline{\Omega})$ there exists a constant $N$ such that the following estimate holds:
   \begin{equation} \label{1.56}
    	\int_{G_{c^2}} s(x)^2 \mathcal{C}^2 \dif x \dif t
		\leq N \int_{Q_\xi} s(x)^2 \mathcal{C}^2 \dif x \dif t.
   \end{equation}
\end{lem}

\begin{proof}
	For each $x \in \Omega$, the cross-section of $G_{c^2}$ in the time direction is the interval $(0,t_{c^2}(x))$ with
	\[
	t_{c^2}(x) = \sqrt{\frac{|x - x_0|^2 - c^2}{\beta}},
	\]
	and the function $\theta(t)=\exp(-2\lambda\beta t^2)$ is decreasing for $t>0$. Since $Q_\xi\subset G_{c^2+3\sigma}\subset G_{c^2}$, we have $t_{c^2}(x)>\xi$ for all $x\in\Omega$. Using that $\theta$ is decreasing, we have $\int_0^{t}\theta(\tau)\dif\tau \le \frac{t}{\xi}\int_0^{\xi}\theta(\tau)\dif\tau$ for every $t\ge\xi$. Consequently, with $t_{c^2}(x)\le T$,
	\begin{align*}
	\int_{G_{c^2}} s(x)^2 \mathcal{C}^2 \dif x \dif t
	&= \int_{\Omega} s(x)^2 e^{2\lambda|x-x_0|^2} \Bigl( \int_0^{t_{c^2}(x)} e^{-2\lambda\beta t^2} \dif t \Bigr) \dif x \\
	&\le \int_{\Omega} s(x)^2 e^{2\lambda|x-x_0|^2} \frac{t_{c^2}(x)}{\xi} \Bigl( \int_0^{\xi} e^{-2\lambda\beta t^2} \dif t \Bigr) \dif x \\
	&\le \frac{T}{\xi} \int_{Q_\xi} s(x)^2 \mathcal{C}^2 \dif x \dif t,
	\end{align*}
	which is \eqref{1.56} with $N=T/\xi$. The proof is done.
\end{proof}

After presenting a basic lemma, we turn to the discussion of energy estimate.
Set $\Omega_0:=\{0\}\times\Omega$ and $\Omega_T:=\{T\}\times\Omega$.

\begin{lem}\label{lemma-3.2}
    $w\in H^2(Q_T)$ satisfy the following hyperbolic equation, 
    \begin{equation*} 
    \left\{\begin{aligned}
        p_1^{-1}w_{tt} &=\Delta w +f, && (t,x) \in Q_T , \\
        w &= g_t, && (t,x) \in \Sigma_T , \\
    \end{aligned}\right.
	\end{equation*}
where $p_1^{-1}\in L^\infty(\Omega)$, $f\in L^2(Q_T)$, $g_t\in H^1(\Sigma_T)$, and $h_t=\partial_\nu w\in L^2(\Sigma_T)$ denotes the normal derivative on $\Sigma_T$. Suppose that $w$ is a solution of this system and that $\mathcal{C}$ is a Carleman weight function. Then there exist positive constants $M',N$ such that
	\begin{equation*}
		\begin{aligned}
	        \int_{\Omega_0} (|\nabla w|^2 + p_1^{-1} w_t^2 + M'\lambda^2 w^2)\mathcal{C}^2 \dif x &\leq \int_{\Omega_T} (|\nabla w|^2 + p_1^{-1} w_t^2 + M'\lambda^2 w^2)\mathcal{C}^2 \dif x \\
	        &+N\int_{Q_T} \bigl(\lambda |\nabla  w|^2 +\lambda w_t^2+ \lambda^3 w^2 +f^2 \bigr)\mathcal{C}^2 \dif x \dif \tau \\
	        &+ N\int_{\Sigma_T} \bigl(\lambda^2 g_t^2  + g_{tt}^2 + h_t^2 \bigr) \mathcal{C}^2 \dif S \dif \tau.
	    \end{aligned}
	\end{equation*}
\end{lem}

\begin{proof}
	We introduce $\tilde{w} = \mathcal{C}w$.
    Consider
    \begin{equation}\label{e}
		e_{\tilde{w}}(t):=\frac{1}{2}\int_{\Omega} (|\nabla\tilde{w}|^2 + p_1^{-1}\tilde{w}_t^2 + M'\lambda^2 \tilde{w}^2) \dif x, \quad M' = 4d^2+2,
	\end{equation}
	 then differentiate \eqref{e} with respect to $t$, we have
	\begin{equation*}
    \begin{aligned}
        e'_{\tilde{w}}(t) &=\int_{\Omega} \bigl(\nabla\tilde{w}\cdot\nabla\tilde{w}_t + p_1^{-1}\tilde{w}_t\tilde{w}_{tt} + M'\lambda^2\tilde{w}\tilde{w}_t\bigr) \dif x \\
        &=\int_{\Omega} \bigl(-\Delta\tilde{w} + p_1^{-1}\tilde{w}_{tt} + M'\lambda^2\tilde{w}\bigr)\tilde{w}_t \dif x+\int_{\partial\Omega} \frac{\partial\tilde{w}}{\partial \nu}\tilde{w}_t \dif S.\notag
    	\end{aligned}
	\end{equation*}
Integrating the above equality over $(0,t)$, we obtain
\begin{equation}\label{912+}
    \begin{aligned}
        e_{\tilde{w}}(t)-e_{\tilde{w}}(0)&=\int_0^t e'_{\tilde{w}}(\tau) \dif \tau\\
        &=\int_0^t \int_{\Omega} \bigl(-\Delta\widetilde{w} + p_1^{-1}\widetilde{w}_{tt} + M'\lambda^2\widetilde{w}\bigr)\widetilde{w}_t \dif x \dif \tau+\int_0^t \int_{\partial\Omega} \frac{\partial\widetilde{w}}{\partial \nu}\widetilde{w}_t \dif S \dif \tau.
    \end{aligned}
\end{equation}
Since \(\widetilde w=\mathcal Cw\), we have
\[
\widetilde w_t=\mathcal C_t w+\mathcal Cw_t,
\qquad
\nabla\widetilde w
=\mathcal C\nabla w+w\nabla\mathcal C,
\]
\[
\widetilde w_{tt}
=
\mathcal Cw_{tt}
+2\mathcal C_tw_t
+\mathcal C_{tt}w,\qquad \Delta\widetilde w
=
\mathcal C\Delta w
+2\nabla\mathcal C\cdot\nabla w
+w\Delta\mathcal C.
\]
Consequently,
\begin{equation*}
	\begin{aligned}
-\Delta\widetilde w
+p_1^{-1}\widetilde w_{tt}
+M'\lambda^2\widetilde w
={}&
\mathcal C\left(-\Delta w+p_1^{-1}w_{tt}\right)
-2\nabla\mathcal C\cdot\nabla w
-(\Delta\mathcal C)w\\ &
+2p_1^{-1}\mathcal C_tw_t
+p_1^{-1}\mathcal C_{tt}w
+M'\lambda^2\mathcal Cw.
\end{aligned}
\end{equation*}
Using the equation $p_1^{-1}w_{tt}=\Delta w+f,$ we obtain
\[
-\Delta\widetilde w
+p_1^{-1}\widetilde w_{tt}
+M'\lambda^2\widetilde w
={}
\mathcal C f
-2\nabla\mathcal C\cdot\nabla w
-(\Delta\mathcal C)w
+2p_1^{-1}\mathcal C_tw_t
+p_1^{-1}\mathcal C_{tt}w
+M'\lambda^2\mathcal Cw.
\]
For the Carleman weight function \(\mathcal C\), we have
\[
|\mathcal C_t|+|\nabla\mathcal C|
\leq N\lambda\mathcal C,
\qquad
|\Delta\mathcal C|+|\mathcal C_{tt}|
\leq N\lambda^2\mathcal C.
\]
Hence
$$
\begin{aligned}
\left|
-\Delta\widetilde w
+p_1^{-1}\widetilde w_{tt}
+M'\lambda^2\widetilde w
\right|\leq
N\mathcal C
\left(
|f|+\lambda|\nabla w|
+\lambda|w_t|+\lambda^2|w|
\right).
\end{aligned}
$$
On the other hand,
$$
|\widetilde w_t|
\leq
|\mathcal C_t||w|+\mathcal C|w_t|
\leq
N\mathcal C
\left(\lambda|w|+|w_t|\right).
$$
Therefore, by Young's inequality,
$$
\begin{aligned}
\left|
\left(
-\Delta\widetilde w
+p_1^{-1}\widetilde w_{tt}
+M'\lambda^2\widetilde w
\right)\widetilde w_t
\right|&\leq
N\mathcal C^2
\left(
|f|+\lambda|\nabla w|
+\lambda|w_t|+\lambda^2|w|
\right)
\left(
|w_t|+\lambda|w|
\right)
\\
&\leq
N\mathcal C^2
\left(
f^2+\lambda|\nabla w|^2
+\lambda w_t^2+\lambda^3w^2
\right).
\end{aligned}
$$
Thus,
$$
\begin{aligned}
\left|
\int_0^t\int_\Omega
\left(
-\Delta\widetilde w
+p_1^{-1}\widetilde w_{tt}
+M'\lambda^2\widetilde w
\right)\widetilde w_t
\,dx\,d\tau
\right|\leq&
N\int_0^t\int_\Omega
\left(
\lambda|\nabla w|^2
+\lambda w_t^2
+\lambda^3w^2
\right)\mathcal C^2\,dx\,d\tau
\\
&+
N\int_0^t\int_\Omega
f^2\mathcal C^2\,dx\,d\tau.
\end{aligned}
$$

It remains to estimate the boundary term. Since
\[
w=g_t\qquad\text{on }\Sigma_T.
\]
Similarly, using the corresponding boundary condition for the normal derivative, we have
$$
|\partial_\nu\widetilde w|^2
\leq
N\mathcal C^2
\left(
h_t^2+\lambda^2g_t^2
\right),
$$
It follows from Young's inequality that
$$
\begin{aligned}
\left|
\int_0^t\int_{\partial\Omega}
\frac{\partial\widetilde w}{\partial\nu}
\widetilde w_t\,dS\,d\tau
\right|\leq
N\int_0^t\int_{\partial\Omega}
\left(
\lambda^2g_t^2
+g_{tt}^2+h_t^2
\right)\mathcal C^2\,dS\,d\tau.
\end{aligned}
$$
Combining the above estimates with the energy identity \eqref{912+}, we conclude that
\begin{align*}
    e_{\widetilde{w}}(0)
    & \leq e_{\widetilde{w}}(t)+N\int_0^t \int_{\Omega} \bigl(\lambda |\nabla w|^2 +\lambda w_t^2+ \lambda^3 w^2\bigr)\mathcal{C}^2 \dif x \dif \tau \\
    & \quad + N\int_0^t \int_{\Omega} f^2 \mathcal{C}^2 \dif x \dif \tau + N\int_0^t \int_{\partial\Omega} \bigl(\lambda^2g_t^2 + g_{tt}^2 + h_t^2\bigr) \mathcal{C}^2 \dif S \dif \tau.\notag
\end{align*}
Let \,$t=T$\,, then we have
\begin{equation*}
    \begin{aligned}
        \int_{\Omega_0} (|\nabla w|^2 + p_1^{-1} w_t^2 + M'\lambda^2 w^2)\mathcal{C}^2 \dif x &\leq \int_{\Omega_T} (|\nabla w|^2 + p_1^{-1} w_t^2 + M'\lambda^2 w^2)\mathcal{C}^2 \dif x \\
        &+N\int_{Q_T} \bigl(\lambda |\nabla  w|^2 +\lambda w_t^2+ \lambda^3 w^2 +f^2 \bigr)\mathcal{C}^2 \dif x \dif \tau \\
        &+ N\int_{\Sigma_T} \bigl(\lambda^2 g_t^2 + g_{tt}^2 + h_t^2\bigr) \mathcal{C}^2 \dif S \dif \tau.
    \end{aligned}
\end{equation*}
The proof is done.
\end{proof}

To complete the proof of our main theorem, the following lemma derived from a Carleman estimate is indispensable.
\begin{lem} \label{lemma-3.3}
     Suppose $w\in C^2(Q_T)$ is the solution of the hyperbolic equation \eqref{w} (so that $w_t(0,x)=0$ by \eqref{eq:wt0}), $\mathcal{C},\psi$ is defined as in \eqref{CWF}, and the parameters $\beta,a,M,d$ satisfy \eqref{eq:betaband}. Then there exists a constant $N$ such that the following Carleman estimate holds:
    \begin{align*}
		&\int_{G_{c^2}} \lambda(|\nabla_{t,x} w|^2 + \lambda^2 w^2) \mathcal{C}^2 \dif x\dif \tau \\
		\leq & \, N\int_{G_{c^2}} \bigl(|\nabla_{t,x} w|^2 +w^2+|\Delta q|^2 + |\nabla q|^2 + q^2\bigr) \mathcal{C}^2 \dif x\dif \tau\\
		&+N\int_{\partial_1 G_{c^2}} (\lambda|\nabla_{t,x} w|^2+\lambda^3 w^2) \mathcal{C}^2 \dif S \dif \tau+ N\int_{\Sigma_T} (\lambda|\nabla_{t,x} w|^2+\lambda^3 w^2) \mathcal{C}^2 \dif S \dif \tau.
	\end{align*}
\end{lem}

\begin{proof}
    Suppose $a^{jk}(x)$ be a real symmetric matrix-valued, $C^2$-smooth function and let $\widetilde{\ell}$ is $C^2$-smooth. Since $\psi_{t}=-2\beta t$, $\psi_{tt}=-2\beta$, $\psi_i=2(x_i-x_{0i})$, $\psi_{ii}=2$, $\psi_{ij}=2\delta_{ij}$, then applying Proposition \ref{prop-2511} to $\tilde{w}=\mathcal{C}w$ with $a^{tt}=p_1^{-1}$, $a^{ii}=-1$ and $\ell=\lambda\psi$, we obtain
	\begin{equation}\label{ell}
\begin{aligned}
e^{2\lambda\psi}
\left|
p_1^{-1}w_{tt}-\Delta_x w
\right|^2
&\geq(-8\lambda\beta p_1^{-2}
-2p_1^{-1}\widetilde{\ell}
+4\lambda(x-x_0)\cdot\nabla(p_1^{-1}))\widetilde{w}_t^2
\\
&+(
8\lambda+2\widetilde{\ell})
\left|\nabla\widetilde{w}\right|^2
-2p_1^{-1}\widetilde{\ell}_t
\widetilde{w}\widetilde{w}_t
+
2[
2\lambda\beta\nabla(p_1^{-1})
+\nabla\widetilde{\ell}
]\cdot
\widetilde{w}\nabla\widetilde{w}
\\
&+
D\widetilde{w}^2
+\partial_t F^t
+
\sum_{j=1}^{n}
\partial_{j}F^j,
\end{aligned}
\end{equation}
\begin{equation*}\left\{
\begin{aligned}
F^t&={}4\lambda\beta t\,p_1^{-2}\tilde w_t^2
+4\lambda\beta\,p_1^{-2}\tilde w_t\tilde w
+\left(
16\lambda^3\beta^3t^3
+8\lambda^2\beta^2t
\right)p_1^{-2}\tilde w^2
\\
&
+4\lambda\beta t\,p_1^{-1}
\left|\nabla\tilde w\right|^2
+8\lambda p_1^{-1}
\bigl((x-x_0)\cdot\nabla\tilde w\bigr)\tilde w_t+
(
4n\lambda p_1^{-1}
+2\tilde\ell p_1^{-1}
)\tilde w_t\tilde w
\\
&
+\left[
-16\lambda^3\beta t\,p_1^{-1}|x-x_0|^2
+8n\lambda^2\beta t\,p_1^{-1}
\right]\tilde w^2,\\
F^j&={}-4\lambda p_1^{-1}(x_j-x_{0,j})\tilde w_t^2
-8\lambda\beta t\,p_1^{-1}\tilde w_j\tilde w_t
+4\lambda(x_j-x_{0,j})
\left|\nabla\tilde w\right|^2
\\&
-8\lambda
\bigl((x-x_0)\cdot\nabla\tilde w\bigr)\tilde w_j
\\
&
+[
16\lambda^3\beta^2t^2p_1^{-1}(x_j-x_{0,j})
+8\lambda^2\beta t\,p_1^{-1}(x_j-x_{0,j})
-16\lambda^3|x-x_0|^2(x_j-x_{0,j})\\ &
+8n\lambda^2(x_j-x_{0,j})
]\tilde w^2-(
4\lambda\beta p_1^{-1}
+4n\lambda
+2\tilde\ell
)\tilde w_j\tilde w,
\qquad j=1,\ldots,n.
\end{aligned}
\right.
\end{equation*}
where
\begin{align*}
D
={}&
-32\lambda^3\beta^3p_1^{-2}t^2
+
32\lambda^3|x-x_0|^2
-16\lambda^3\beta^2t^2
\nabla(p_1^{-1})\cdot(x-x_0)\\ &-8\lambda^2\beta
\nabla(p_1^{-1})\cdot(x-x_0)
-8\lambda^2\beta^2p_1^{-2}
+8n\lambda^2
-8\lambda^2\beta p_1^{-1}
+8\lambda^2\beta^2t\,\nabla(p_1^{-1})\cdot(x-x_0)
\\ &+8\lambda^2\widetilde{\ell}
[\beta^2p_1^{-1}t^2
-
|x-x_0|^2]
-4\lambda
\left(
\beta p_1^{-1}+n
\right)
\widetilde{\ell}
-
\widetilde{\ell}^{\,2}.
\end{align*}
	Choose $\widetilde{\ell} = -8\lambda \beta $, which is independent of $x$ and $t$, so that $\widetilde{\ell}_t=\nabla\widetilde{\ell}=0$; hence the terms $-2p_1^{-1}\widetilde{\ell}_t\tilde w\tilde w_t$ and $2\nabla \widetilde{\ell}\cdot\tilde w\nabla \tilde{w}$ in \eqref{ell} vanish, while the cross term $4\lambda\beta\nabla(p_1^{-1})\cdot\tilde w\nabla\tilde w$ remains. Substituting $\widetilde{\ell}$ into \eqref{ell} and using $a^{-1}\leq p_1\leq 1$, the condition $(x-x_0,\nabla p_1)\le 0$ in \eqref{1.2}, and the bound $\lvert\nabla(p_1^{-1})\rvert=p_1^{-2}\lvert\nabla p_1\rvert\le a^2M$, we absorb the remaining cross term by Young's inequality:
\begin{equation*}
\begin{aligned}
4\lambda\beta\nabla(p_1^{-1})\cdot\tilde w\nabla\tilde w
&\ge -4\lambda\beta\,\lvert\nabla(p_1^{-1})\rvert\,\lvert\tilde w\rvert\,\lvert\nabla\tilde w\rvert= -2\bigl(2\beta\lvert\nabla(p_1^{-1})\rvert\,\lvert\nabla\tilde w\rvert\bigr)\bigl(\lambda\lvert\tilde w\rvert\bigr)\\
&\ge -\frac{8\beta^2}{a^4}\lvert\nabla(p_1^{-1})\rvert^2\lvert\nabla\tilde w\rvert^2-\frac{\lambda^2a^4}{2}\tilde w^2\\
&\ge -8\beta^2M^2\lvert\nabla\tilde w\rvert^2-\frac{\lambda^2a^4}{2}\tilde w^2,
\end{aligned}
\end{equation*}
where in the third line we used $2AB\le\varepsilon A^2+B^2/\varepsilon$ with $A=2\beta\lvert\nabla(p_1^{-1})\rvert\lvert\nabla\tilde w\rvert$, $B=\lambda\lvert\tilde w\rvert$, $\varepsilon=2/a^4$, and in the last line $\lvert\nabla(p_1^{-1})\rvert\le a^2M$. Suppose that the parameters satisfy
	\begin{equation}\label{eq:betaband}
		\beta <\min \{\frac{1}{4(1+M)}, \frac{M\sqrt d}{64}\},\qquad a<2
	\end{equation}
	Moreover, in $G_{c^2}$ we have $|x-x_0|^2>c^2+\beta t^2$, then
	\begin{equation*}
	\begin{aligned}
	D\ge{}& 16\lambda^3(c^2+2\beta t^2)
	-\lambda^2C(\beta,a,M,d,n),
	\end{aligned}
	\end{equation*}
	Finally, the coefficient of $\tilde w_t^2$ in \eqref{ell} equals $-8\lambda\beta p_1^{-2}-2p_1^{-1}\widetilde\ell+4\lambda(x-x_0)\cdot\nabla(p_1^{-1})=8\lambda\beta p_1^{-2}(2p_1-1)-4\lambda p_1^{-2}(x-x_0)\cdot\nabla p_1$. Since $(x-x_0)\cdot\nabla p_1\le 0$ and $p_1^{-2}(2p_1-1)\ge a(2-a)$ for $p_1\in[a^{-1},1]$ with $a<2$, this coefficient is bounded from below by $8\lambda\beta a(2-a)$. Therefore
	\begin{equation*}
	\begin{aligned}
	e^{2\lambda\psi}\bigl|p_1^{-1}w_{tt}-\Delta_x w\bigr|^2
	&\ge (8\lambda \beta a(2-a))\tilde w_t^2
	+[8\lambda (1-2\beta )-8\beta^2M^2]|\nabla\tilde w|^2\\
	&\quad + \bigl(16\lambda^3(c^2+2\beta t^2)-\lambda^2C(\beta,a,M,d,n)\bigr)\tilde w^2+\partial_t F^t+\sum_{j=1}^{n}\partial_jF^j,
	\end{aligned}
	\end{equation*} 
	Integrate the above pointwise estimate \eqref{ell} over $G_{c^2}$, we obtain
	\begin{align*}
		&\int_{G_{c^2}} \lambda( |\nabla_{t,x} w|^2 + \lambda^2 w^2) \mathcal{C}^2 \dif x \dif \tau \\
		\leq & \, N\int_{G_{c^2}} \bigl(|\nabla_{t,x} w|^2 +w^2+|\Delta q|^2 + |\nabla q|^2 + q^2\bigr) \mathcal{C}^2 \dif x \dif \tau\\
		& + N \int_{\partial_1 G_{c^2}} (\lambda|\nabla_{t,x} w|^2+\lambda^3 w^2) \mathcal{C}^2 \dif S \dif \tau+ N\int_{\Sigma_T} (\lambda|\nabla_{t,x} w|^2+\lambda^3 w^2) \mathcal{C}^2 \dif S \dif \tau.
	\end{align*}
	Here the boundary integral over $\{t=0\}\times\Omega$ vanishes: since $w_t(0,x)=0$ by \eqref{eq:wt0}, every term of $F^t|_{t=0}$ contains the factor $\tilde w_t(0,\cdot)=\mathcal C(0,\cdot)w_t(0,\cdot)=0$, so that $\int_{\Omega_0}F^t|_{t=0}\,\dif x=0$. The boundary integrals over $\partial_1G_{c^2}$ and $\Sigma_T$ are estimated by the explicit formulas of $F^t$ and $F^j$; each term of $F^t$, $F^j$ is a combination of $\lambda\beta t$-, $\lambda$- and $\lambda^3$-weighted quadratics in $\tilde w$, $\nabla\tilde w$, hence it is bounded by $C(\lambda|\nabla_{t,x}\tilde w|^2+\lambda^3\tilde w^2)$, which gives the last two terms of the right-hand side after rewriting $\tilde w=\mathcal C w$. We obtain the result.
\end{proof}
In order to further address the terms involving the initial values in our proof, the following lemma must be proved first, which is mainly based on Lemma \ref{lemma-4}.
\begin{lem} \label{lemma-3.4}
    Suppose the following expression is valid for $v\in C^2(Q_T)$
	\begin{align}
		\bar{v}_t(0,x)^2 \geq& \frac{1}{2}[p_1\Delta (p_1^2\Delta \hat{f})]^2-4[\Delta(\nabla q \cdot\nabla f_2)]^2-2N(\sum_{i,j=1}^n |q_{ij}|+|\nabla q|+|q|)^2\label{eq:vt0}\\
                &-2N(\sum_{i,j=1}^n |\hat{f}_{ij}|+|\nabla \hat{f}|+|\hat{f}|)^2,\quad x\in \Omega,\nonumber
	\end{align}
    where $f_2, \hat{f}\in C^4(\overline{\Omega})$ and $q=p_1-p_2$, $p_1,p_2\in W(a,M,x_0)$. Suppose condition \eqref{1.901} holds, then we have the inequality
    \begin{equation*}
    \begin{aligned}
        \lambda \beta \int_{G_{c^2+2\sigma}} v_t^2 \mathcal{C}^2 &\geq \frac{\lambda\beta }{8}\int_{G_{c^2+2\sigma}} \bigl|p_1\Delta (p_1^2\Delta \hat{f})\bigr|^2\mathcal{C}^2- N\lambda \beta \int_{Q_\xi} ( \sum_{i,j=1}^n q_{ij}^2 + |\nabla q|^2 + q^2 ) \mathcal{C}^2 \\
        &- \int_{G_{c^2+2\sigma}} w_t^2 \mathcal{C}^2- N\lambda \beta \int_{Q_\xi} ( \sum_{i,j=1}^n \hat{f}_{ij}^2+|\nabla \hat{f}|^2+\hat{f}^2 ) \mathcal{C}^2.
    \end{aligned}
\end{equation*}
where $\mathcal{C}$ and $\beta$ are defined as in \eqref{CWF}.
\end{lem}

\begin{proof}
Consider 
\begin{align*}
v_t(t,x) &= v_t(0,x) + \int_0^t v_{tt}(\tau,x) \dif \tau = v_t(0,x) + \int_0^t w_t(\tau,x) \dif \tau,
\end{align*}
hence,
\begin{align*}
v_t^2(t,x) \geq \frac{1}{2} v_t^2(0,x) - \big(\int_0^t w_t(\tau,x) \dif \tau \big)^2.
\end{align*}
Using \eqref{eq:vt0}, we can continue the inequality as follows,
\begin{equation*}
    \begin{aligned}
	v_t^2(t,x) &\geq \frac{1}{4}\bigl|p_1\Delta (p_1^2\Delta \hat{f})\bigr|^2
	-2 \bigl| \Delta(\nabla q \cdot \nabla f_2) \bigr|^2 - N( \sum_{i,j=1}^n q_{ij}^2 + |\nabla q|^2 + q^2 ) \\
    &-N(\sum_{i,j=1}^n \hat{f}_{ij}^2+|\nabla \hat{f}|^2+\hat{f}^2)- ( \int_0^t w_\tau (\tau,x) \dif \tau )^2.
\end{aligned}
\end{equation*}
Using Condition \eqref{1.901}, $\bigl| \Delta(\nabla q \cdot \nabla f_2) \bigr|^2$ can be absorbed,
\begin{equation*}
    \begin{aligned}
	v_t^2(t,x) &\geq \frac{1}{8}\bigl|p_1\Delta (p_1^2\Delta \hat{f})\bigr|^2- N( \sum_{i,j=1}^n q_{ij}^2 + |\nabla q|^2 + q^2 )\\
	&-N(\sum_{i,j=1}^n \hat{f}_{ij}^2+|\nabla \hat{f}|^2+\hat{f}^2)- ( \int_0^t w_\tau (\tau,x) \dif \tau )^2.
\end{aligned}
\end{equation*}
Multiplying both sides of the inequality by $\mathcal{C}^2$ and integrating over $G_{c^2+2\sigma}$, we obtain
\begin{equation*}
    \begin{aligned}
    \int_{G_{c^2+2\sigma}} &v_t^2(t,x)\mathcal{C}^2 \geq \frac{1}{8}\int_{G_{c^2+2\sigma}}\bigl|p_1\Delta (p_1^2\Delta \hat{f})\bigr|^2\mathcal{C}^2
    - \int_{G_{c^2+2\sigma}}( \int_0^t w_\tau (\tau,x) \dif \tau )^2 \mathcal{C}^2\\
    &- N\int_{G_{c^2+2\sigma}}( \sum_{i,j=1}^n q_{ij}^2 + |\nabla q|^2 + |q|^2 )\mathcal{C}^2 -N\int_{G_{c^2+2\sigma}}(\sum_{i,j=1}^n \hat{f}_{ij}^2+|\nabla \hat{f}|^2+|\hat{f}|^2)\mathcal{C}^2.
\end{aligned}
\end{equation*}
Using Lemma \ref{lemma-2}
\begin{align}
\int_{G_{c^2+2\sigma}} ( \int_0^t w_\tau (\tau,x) \dif \tau )^2 \mathcal{C}^2 \dif x \dif t 
\leq \frac{1}{\lambda \beta} \int_{G_{c^2+2\sigma}} w_\tau ^2(\tau,x) \mathcal{C}^2 \dif x \dif t,
\end{align}
according to the relationship between $Q_\xi$ and $G_{c^2+2\sigma }$, we obtain
\begin{equation*}
    \begin{aligned}
        \lambda \beta \int_{G_{c^2+2\sigma}} v_t^2 \mathcal{C}^2 &\geq \frac{\lambda\beta }{8}\int_{G_{c^2+2\sigma}} \bigl|p_1\Delta (p_1^2\Delta \hat{f})\bigr|^2\mathcal{C}^2 - N\lambda \beta \int_{Q_\xi} ( \sum_{i,j=1}^n q_{ij}^2 + |\nabla q|^2 + |q|^2 ) \mathcal{C}^2 \\
        &-  \int_{G_{c^2+2\sigma}} w_t^2 \mathcal{C}^2- N\lambda \beta \int_{Q_\xi} ( \sum_{i,j=1}^n \hat{f}_{ij}^2+|\nabla \hat{f}|^2+|\hat{f}|^2 ) \mathcal{C}^2.
    \end{aligned}
\end{equation*}
The proof is done.
\end{proof}

\section{The proof of Theorem \ref{thm}} \label{sec:proof-SS26}

Introduce the cut-off function $\chi_\sigma(t, x) \in C^3(\overline{Q}_T)$ such that
\begin{equation}\label{1.33}
    \chi_\sigma(t, x) =
    \begin{cases}
        1, & \text{if } (t, x) \in G_{c^2 + 2\sigma}, \\
        0, & \text{if } (t, x) \in Q_T \backslash G_{c^2 + \sigma}, \\
        \text{between } 0 \text{ and } 1, & \text{otherwise}.
    \end{cases}
\end{equation}
Denote $\bar{v}(t, x) = \chi_\sigma(t, x) v(t, x)$, $v$ is the solution to \eqref{1.4}. Then the definition of $\partial_{1} G_{c^2}$ and \eqref{1.33} imply that
\begin{equation} \label{1.34}
    \bar{v}\big|_{\partial_1 G_{c^2}} =\bar{v}_t\big|_{\partial_1 G_{c^2}}=\nabla \bar{v} \big|_{\partial_1 G_{c^2}} = 0.
\end{equation}
By the definition of $Q_\xi$ and \eqref{1.33}, $\chi_\sigma(t, x) = 1$ for $Q_\xi$. Hence, $\bar{v}(0, x) = v(0, x)$ and $\bar{v}_t(0, x) = v_t(0, x)$.
Recall the notations $\mathcal L_p$, $F[p]$ and $\mathcal A(p_1, p_2)$ defined in \eqref{eq:calL} and \eqref{eq:calA}.
Multiplying \eqref{1.4} by the function $\chi_\sigma(t, x)$, we obtain
\begin{equation} \label{1.35}
	\left\{\begin{aligned}
		\partial_{t}^2 \bar{v}=&(\mathcal L_{p_1} + F[p_1])\bar{v} + \chi_\sigma\mathcal{A}(p_1,p_2) + [\chi_\sigma,\mathcal L_{p_1} + F[p_1]] v, && (t,x) \in Q_T,\\
		\bar{v}=& \chi_\sigma\tilde{g},&&(t,x)\in \Sigma_T,\\
		\bar{v}(0,x)=&0,&&x\in \Omega,\\
		\bar{v}_t(0,x)=&(\mathcal L_{p_1} + F[p_1])^2 (f_1p_1)-(\mathcal L_{p_2} + F[p_2])^2 (f_2p_2),&&x\in \Omega,
	\end{aligned}\right.
\end{equation}
where $[\chi_\sigma,\mathcal L_{p_1} + F[p_1]]$ signifies the commutator between $\chi_\sigma$ and $\mathcal L_{p_1} + F[p_1]$.
There holds
\[
[\chi_\sigma,\mathcal L_{p_1} + F[p_1]]
= 2\partial_t \chi_{\sigma}\partial_t + \partial_t^2\chi_{\sigma } - 2p_1\nabla\chi_{\sigma} \nabla + \nabla p_1\nabla \chi_{\sigma}-p_1\Delta \chi_{\sigma}.
\]
Also, denote $\bar{w}=\chi_{\sigma} w(t,x)$, $w=\partial_t v$ is the solution to \eqref{w}. Then we have 
\begin{equation}
    \left\{\begin{aligned}\label{1.39}
        \partial_t^2 \bar{w} =& (\mathcal L_{p_1} + F[p_1]) \bar{w}+\chi_\sigma \partial_t \mathcal{A}(p_1,p_2) + [\chi_\sigma,\mathcal L_{p_1} + F[p_1]] w, && (t,x) \in Q_T,\\
        \bar{w}=&\chi_{\sigma} \tilde{g}_t,  &&(t,x)\in \Sigma_T,\\
        \bar{w}(0,x)=&(\mathcal L_{p_1} + F[p_1])^2 (f_1p_1)-(\mathcal L_{p_2} + F[p_2])^2 (f_2p_2),&&x\in \Omega,\\
        \bar{w}_t(0,x)=&\mathcal{A}(p_1,p_2),&&x\in \Omega.
    \end{aligned}\right.
\end{equation}
To elaborate the relation between the difference $p_1 - p_2$ and $\mathcal{A}(p_1,p_2)$, we express $\mathcal{A}(p_1,p_2)$ in terms of $q := p_1 - p_2$ as 
\begin{equation} \label{mathcalA}
   \mathcal{A}(p_1,p_2)= h^{(2)} \Delta q + \sum_{i=1}^{n} h^{i} q_i + h^{(0)} q,
\end{equation}
We arrange the expansion of $\mathcal D$ in descending order of the differentiation orders with respect to $q$ and $\hat{f}$.
\begin{align}
	\mathcal{D}(p_1,p_2,f_1,f_2)
	:= & \, (\mathcal L_{p_1} + F[p_1])^2 (f_1p_1)-(\mathcal L_{p_2} + F[p_2])^2 (f_2p_2) \nonumber \\
	= & \, p_1\Delta (p_1^2\Delta \hat{f}) +p_1p_2\Delta(\nabla q\cdot\nabla f_2)+\sum_{i,j=1}^{n}\alpha_{ij}(x)q_{ij} \nonumber \\
	& + \sum_{i=1}^{n}\alpha_{(1i)}q_i + \alpha_0 q +\sum_{i,j=1}^{n}\beta_{ij}(x)\hat{f}_{ij}+\sum_{i=1}^{n}\beta_{(1i)} \hat{f}_i+\beta_0 \hat{f}, \label{mathcalD}
\end{align}
where $q_{ij}$, $q_i$, $\hat{f}_{ij}$, and $\hat{f}_i$ denote the corresponding partial derivatives with respect to the $x$-components.
The coefficients $h^{(2)}$, $h^i$, $h^{(0)}$, $\alpha_{ij}$, $\alpha_{(1i)}$, $\alpha_0$, $\beta_{ij}$, $\beta_{(1i)}$, $\beta_0$ are bounded functions, each of which is given by a polynomial combination of derivatives, up to order 2, of $\partial_t^3U_2$, $p_1$, $p_2$, $f_1$ and $f_2$. More precisely, $h^{(2)}\in C^3(\overline{Q_T})$, $h^{i}\in C^2(\overline{Q_T})$, $h^{(0)}\in C^1(\overline{Q_T})$, while the $\alpha$- and $\beta$-coefficients depend only on $x$ and belong to $C^3(\overline{\Omega})$. In view of the assumptions of Theorem \ref{thm} and \eqref{1.2}, all of these coefficients are bounded by a constant $N$. Consequently, the terms involving $p_1,p_2,f_1$ and $f_2$ through these coefficients can be controlled and absorbed into the generic constant $N$, and they do not affect the essential estimates for $q$ and $\hat{f}$.

Because $\partial_t\chi_{\sigma }=\partial_t^2\chi_{\sigma }=\partial_{x_i}\chi_{\sigma }=\partial_{x_ix_j}\chi_{\sigma }=0$ in $G_{c^2+2\sigma}$ by \eqref{1.33}, from \eqref{1.35} we have
\begin{equation}\label{1.43}
    \left\{\begin{aligned}
        |p_1^{-1}\bar{v}_{tt} - \Delta \bar{v} |&\leq N(|\nabla_{t,x} \bar{v}| + |\bar{v}|+|\Delta q|+|\nabla q|+|q|)\\
		&+N\mathbf 1_{G_{c^2+\sigma}\setminus G_{c^2+2\sigma}}(|\nabla_{t,x} v|+|v|), && (t,x) \in Q_T,\\
        \bar{v}&=\chi_{\sigma} \tilde{g}, && (t,x)\in \Sigma_T,\\
        \bar{v}(0,x)&=0, && x\in \Omega,
    \end{aligned}\right.
\end{equation}
where $\bar{v}_t(0,x)$ satisfies
\begin{align*}
	\bar{v}_t(0,x)^2 &\geq \frac{1}{2}[p_1\Delta (p_1^2\Delta \hat{f})]^2-4[\Delta(\nabla q\cdot\nabla f_2)]^2\\
	&- N(\sum_{i,j=1}^n |q_{ij}|+|\nabla q|+|q|)^2 - N(\sum_{i,j=1}^n |\hat{f}_{ij}|+|\nabla \hat{f}|+|\hat{f}|)^2,
\end{align*}
and from \eqref{mathcalA} we have
\begin{equation} \label{1.47}
    \left\{\begin{aligned}
        |p_1^{-1}\bar{w}_{tt}-\Delta \bar{w} |&\leq N(|\nabla_{t,x} \bar{w}|+|\bar{w}|+|\Delta q|+|\nabla q|+|q|)\\
        &+N\mathbf 1_{G_{c^2+\sigma}\setminus G_{c^2+2\sigma}}(|\nabla_{t,x} w|+|w|),&& (t,x) \in Q_T,\\
        \bar{w}&=\chi_{\sigma} \tilde{g}_t,\quad \frac{\partial \bar{w}}{\partial \nu}=\chi _{\sigma}\tilde{h}_t+\frac {\partial\chi_{\sigma}}{\partial \nu}\tilde{g}_t,&&(t,x)\in \Sigma_T,\\
        \bar{w}(0,x)&=\mathcal{D}(p_1,p_2,f_1,f_2),&&x\in \Omega,\\
        \bar{w}_t(0,x)&=\mathcal{A}(p_1,p_2),&&x\in \Omega.
    \end{aligned}\right.
\end{equation}
where $\mathbf 1_{G_{c^2+\sigma}\setminus G_{c^2+2\sigma}}$ is a shell indicator function. Alternatively, isolating the $q$-term in \eqref{1.43}, we obtain
\begin{equation}\label{1.432}
	\begin{aligned}
		\bar{v}_t(0,x)^2 &\geq \frac{1}{2}[p_1p_2\Delta(\nabla q\cdot\nabla f_2)]^2-2[\Delta (p_1^2\Delta \hat{f})]^2-N(\sum_{i,j=1}^n |q_{ij}|+|\nabla q|+|q|)^2\\
                &-N(\sum_{i,j=1}^n |\hat{f}_{ij}|+|\nabla \hat{f}|+|\hat{f}|)^2,\quad x\in \Omega.
	\end{aligned}
\end{equation}

We begin by dealing with \eqref{1.43} and \eqref{1.47} to derive two fundamental inequalities.

\begin{lem}
    Under the conditions of Theorem \ref{thm} and the expansion of \eqref{mathcalD}, $v$ satisfies \eqref{1.43} and $w$ satisfies \eqref{1.47}. $\mathcal{C}$, $\beta $ are defined in \eqref{CWF}. There exists a positive constant $N$ such that
    \begin{align}
        \frac{\beta}{2}\lambda \int_{G_{c^2+2\sigma}} &\bigl|p_1\Delta (p_1^2\Delta \hat{f})\bigr|^2 \mathcal{C}^2 \dif x \dif t+\lambda \int_{G_{c^2+2\sigma}} \bigl(|\nabla_{t,x} v|^2 + \lambda^2 v^2\bigr) \mathcal{C}^2 \dif x \dif t \nonumber \\
        &\leq N\int_{G_{c^2+2\sigma}} w_t^2 \mathcal{C}^2 \dif x \dif t+N\int_{\Omega_0}\lambda(\bigl|\nabla_{x,t}v\bigr|^2+\lambda^2v^2)\mathcal{C}^2 \dif x \nonumber \\
		&+N \int_{\Sigma_T} \bigl( \lambda|\nabla_{t,x} \widetilde{g}|^2 + \lambda^3\widetilde{g}^2 + \widetilde{h}^2\bigr) \mathcal{C}^2 \dif S + N \exp\bigl[2\lambda(c^2 + 2\sigma)\bigr] \|v\|_{H^1(Q_T)}^2 \nonumber \\
        & + N \int_{Q_\xi} \bigl(|\Delta q|^2 + |\nabla q|^2 + q^2\bigr) \mathcal{C}^2 \dif x \dif t\\
		&+ N \int_{Q_\xi} \bigl(\sum_{i,j=1}^n |\hat{f}_{ij}|^2+|\nabla \hat{f}|^2+\hat{f}^2 \bigr) \mathcal{C}^2\dif x \dif t. \label{1.73}
	\end{align}
	and 
	\begin{align}
	    \int_{G_{c^2+2\sigma}} \bigl( \lambda & |\nabla_{t,x} w|^2 + \lambda^3 w^2 \bigr) \mathcal{C}^2 \dif x \dif t \leq  N \int_{\Sigma_T} \bigl( \lambda|\nabla_{t,x} \widetilde{g}_t|^2 +\lambda^3 \widetilde{g}_t^2 + \widetilde{h}_t^2\bigr) \mathcal{C}^2 \dif S \nonumber \\
	    & + N \exp\bigl[2\lambda(c^2 + 2\sigma) \bigr] \|w\|_{H^1(Q_T)}^2 + N \int_{Q_\xi} \bigl(|\Delta q|^2 + |\nabla q|^2 + q^2\bigr) \mathcal{C}^2 \dif x \dif t \nonumber \\
	    &+ N \int_{\Omega_0}(\lambda|\nabla_{t,x} w|^2+\lambda^3 w^2) \mathcal C^2 \dif x. \label{3.22}
	\end{align}
\end{lem}

\begin{proof}
Multiplying both sides of the first inequality in \eqref{1.43} by $\mathcal{C}$, squaring and integrating over $G_{c^2}$, and applying Lemma \ref{lemma-1} and \eqref{1.34}, we obtain
\begin{align}
    & N \int_{\Sigma_T}\bigl( \lambda|\nabla_{t,x} \tilde{g}|^2+\lambda^3(1+|\partial_\nu\chi_\sigma|^2)\tilde{g}^2+ \tilde{h}^2\bigr) \mathcal{C}^2 dS + N \int_{G_{c^2} \backslash G_{c^2+2\sigma}} \bigl(|\nabla_{t,x} v|^2 + v^2\bigr) \mathcal{C}^2 \dif x \dif t \notag\\
    & \quad + N \int_{G_{c^2}} \bigl(|\Delta q|^2 + |\nabla q|^2 + q^2\bigr) \mathcal{C}^2 \dif x \dif t+N\int_{\Omega_0}\lambda(\bigl|\nabla_{x,t}v\bigr|^2+\lambda^2v^2)\mathcal{C}^2 \dif x  \notag\\
    & \geq C\lambda \int_{G_{c^2+2\sigma}} \bigl(|\nabla_{t,x} v|^2 + \lambda^2 v^2\bigr) \mathcal{C}^2 \dif x \dif t. \notag
\end{align}
Since $\sigma$ in \eqref{sigma} is fixed and $\chi_\sigma\in C^3$, $\|\partial_\nu\chi_\sigma\|_{L^\infty}\le N$, the term $\partial_\nu\chi_\sigma\tilde g$ is absorbed into the $\tilde g$ term
\begin{align}
    & N \int_{\Sigma_T}\bigl( \lambda|\nabla_{t,x} \tilde{g}|^2+\lambda^3\tilde{g}^2+ \tilde{h}^2\bigr) \mathcal{C}^2 dS + N \int_{G_{c^2} \backslash G_{c^2+2\sigma}} \bigl(|\nabla_{t,x} v|^2 + v^2\bigr) \mathcal{C}^2 \dif x \dif t \notag\\
    & \quad + N \int_{G_{c^2}} \bigl(|\Delta q|^2 + |\nabla q|^2 + q^2\bigr) \mathcal{C}^2 \dif x \dif t+N\int_{\Omega_0}\lambda(\bigl|\nabla_{x,t}v\bigr|^2+\lambda^2v^2)\mathcal{C}^2 \dif x  \label{1.52}\\
    & \geq C\lambda \int_{G_{c^2+2\sigma}} \bigl(|\nabla_{t,x} v|^2 + \lambda^2 v^2\bigr) \mathcal{C}^2 \dif x \dif t. \notag
\end{align}
Similarly, from the first inequality in \eqref{1.47} and using Lemma \ref{lemma-1} we obtain
\begin{align} \label{1.53}
& N \int_{\Sigma_T} \bigl(\lambda |\nabla_{t,x} \widetilde{g}_t|^2 + \lambda^3\widetilde{g}_t^2 + \widetilde{h}_t^2\bigr) \mathcal{C}^2 \dif S + N \int_{G_{c^2} \backslash G_{c^2+2\sigma}} \bigl(|\nabla_{t,x} w|^2 + w^2\bigr) \mathcal{C}^2 \dif x \dif t \notag \\
& \quad + N \int_{G_{c^2}} \bigl(|\Delta q|^2 + |\nabla q|^2 + q^2\bigr) \mathcal{C}^2 \dif x \dif t +N\int_{\Omega_0}\lambda(|\nabla_{t,x} w|^2+\lambda^2 w^2) \mathcal C^2 \dif x\\
& \geq C\lambda \int_{G_{c^2+2\sigma}} \bigl(|\nabla_{t,x} w|^2 + \lambda^2 w^2\bigr) \mathcal{C}^2 \dif x \dif t. \notag
\end{align}
Applying Lemma \ref{lemma-3.1} and \eqref{3.4} to the left-hand side of \eqref{1.52}, we obtain
\begin{equation}\label{1.57}
\begin{aligned}
& N \int_{\Sigma_T} \bigl( \lambda|\nabla_{t,x} \widetilde{g}|^2 +\lambda^3\widetilde{g}^2 + \widetilde{h}^2\bigr) \mathcal{C}^2 \dif S
+ N \exp\bigl[2\lambda(c^2 + 2\sigma)\bigr] \|v\|_{H^1(Q_T)}^2\\
&+ N \int_{Q_\xi} \bigl(|\Delta q|^2 + |\nabla q|^2 + q^2\bigr) \mathcal{C}^2 \dif x \dif t+N\int_{\Omega_0}\lambda(\bigl|\nabla_{x,t}v\bigr|^2+\lambda^2v^2)\mathcal{C}^2 \dif x\\
&\geq C\lambda \int_{G_{c^2+2\sigma}} \bigl(|\nabla_{t,x} v|^2 + \lambda^2 v^2\bigr) \mathcal{C}^2 \dif x \dif t.
\end{aligned}
\end{equation}
Applying Lemma \ref{lemma-3.4} to \eqref{1.57}, we can have
\begin{equation*} 
    \begin{aligned}
        \lambda \int_{G_{c^2+2\sigma}} &\bigl|p_1\Delta (p_1^2\Delta \hat{f})\bigr|^2 \mathcal{C}^2 \dif x \dif t +C\lambda \int_{G_{c^2+2\sigma}} \bigl( |\nabla_{t,x} v|^2 + \lambda^2 v^2 \bigr) \mathcal{C}^2 \dif x \dif t\\
        &\leq N\int_{G_{c^2+2\sigma}} w_t^2 \mathcal{C}^2 \dif x \dif t+N\int_{\Omega_0}\lambda(\bigl|\nabla_{x,t}v\bigr|^2+\lambda^2v^2)\mathcal{C}^2 \dif x\\
        &+ N \exp\bigl[2\lambda(c^2 + 2\sigma)\bigr] \|v\|_{H^1(Q_T)}^2 +N \int_{\Sigma_T} (\lambda|\nabla_{t,x} \widetilde{g}|^2 + \lambda^3\widetilde{g}^2 + \widetilde{h}^2) \mathcal{C}^2 \dif S \\
        &+ N \int_{Q_\xi} \bigl(|\Delta q|^2 + |\nabla q|^2 + q^2\bigr) \mathcal{C}^2 \dif x \dif t+ N \int_{Q_\xi} \bigl(\sum_{i,j=1}^n |\hat{f}_{ij}|^2+|\nabla \hat{f}|^2+\hat{f}^2 \bigr) \mathcal{C}^2\dif x \dif t
    \end{aligned}
\end{equation*}
In the same way, from \eqref{1.53} we can obtain \eqref{3.22}.
\end{proof}

Our next goal is to handle the initial boundary term $\int_{\Omega_0}(\lambda|\nabla_{t,x} w|^2+\lambda^3 w^2) \mathcal C^2$ in \eqref{3.22}.

\begin{lem}
    Under the conditions of Theorem \ref{thm}, suppose that $w$ satisfies equation \eqref{w} and that $\bar{w}$ satisfies equation \eqref{1.39}, $\bar{v}$ satisfies \eqref{1.35}. $\mathcal{C}$ is the Carleman weight function defined in \eqref{CWF}. Then there exists a positive constant $N$ such that the following inequalities holds:
    \begin{equation} \label{3.34}
    \begin{aligned}
        &\int_{\Omega_0} (|\nabla_{x,t} w|^2+ \lambda^2 w^2)\mathcal{C}^2 \dif x \leq N\int_{G_{c^2}\backslash G_{c^2+2\sigma}}\lambda (|\nabla_{t,x} w|^2 +\lambda^2w^2) \mathcal{C}^2 \dif x \dif t\\
		&+ N\int_{G_{c^2}}\bigl(|\Delta q|^2 + |\nabla q|^2 + q^2\bigr) \mathcal{C}^2 \dif x \dif t+ N\int_{\Sigma_T} (\lambda|\nabla_{t,x} w|^2+\lambda^3 w^2) \mathcal{C}^2 \dif S,
    \end{aligned}
\end{equation}
and
\begin{equation} \label{3.341}
    \begin{aligned}
        &\int_{\Omega_0} (|\nabla_{x,t} v|^2+ \lambda^2 v^2)\mathcal{C}^2 \dif x \leq N\int_{G_{c^2}\backslash G_{c^2+2\sigma}}\lambda (|\nabla_{t,x} v|^2 +\lambda^2v^2) \mathcal{C}^2 \dif x \dif t\\
		&+ N\int_{G_{c^2}}\bigl(|\Delta q|^2 + |\nabla q|^2 + q^2\bigr) \mathcal{C}^2 \dif x \dif t+ N\int_{\Sigma_T} (\lambda|\nabla_{t,x} v|^2+\lambda^3 v^2) \mathcal{C}^2 \dif S.
    \end{aligned}
\end{equation}
\end{lem}

\begin{proof}
Applying Lemma \ref{lemma-3.3} to the system of equations related to \eqref{1.39}, we have
\begin{equation*}
    \begin{aligned}
        &\int_{G_{c^2}} \lambda(|\nabla_{t,x} \bar{w}|^2 + \lambda^2 \bar{w}^2) \mathcal{C}^2 \dif x \dif t \leq N\int_{G_{c^2}} \bigl(|\Delta q|^2 + |\nabla q|^2 + q^2\bigr) \mathcal{C}^2 \dif x \dif t\\
        &+ N\int_{G_{c^2}\backslash G_{c^2+2\sigma}} \bigl(|\nabla_{t,x} w|^2 +w^2\bigr) \mathcal{C}^2 \dif x \dif t+ N\int_{\Sigma_T} (\lambda|\nabla_{t,x} w|^2+\lambda^3 w^2) \mathcal{C}^2 \dif S .
    \end{aligned}
\end{equation*}
By the definition of $\chi_{\sigma}$ and the relation between $G_{c^2+2\sigma}$ and $G_{c^2}$, we have
\begin{equation}\label{1.60}
    \begin{aligned}
        &\int_{G_{c^2+2\sigma}} \lambda(|\nabla_{t,x} w|^2 + \lambda^2 w^2) \mathcal{C}^2 \dif x \dif t\leq N\int_{G_{c^2}} \bigl(|\Delta q|^2 + |\nabla q|^2 + q^2\bigr) \mathcal{C}^2 \dif x \dif t\\
        &+ N\int_{G_{c^2}\backslash G_{c^2+2\sigma}} \bigl(|\nabla_{t,x} w|^2 +w^2\bigr) \mathcal{C}^2 \dif x \dif t+ N\int_{\Sigma_T} (\lambda|\nabla_{t,x} w|^2+\lambda^3 w^2) \mathcal{C}^2 \dif S .
    \end{aligned}
\end{equation}
Applying Lemma \ref{lemma-3.2} to \eqref{1.39}, the term on $\Omega_T$ vanishes since $\chi_\sigma(T,\cdot)=0$ by the choice of $T$ and $\beta$,
\begin{equation}\label{1.59}
    \begin{aligned}
        \int_{\Omega_0}& (|\nabla w|^2 +p_1^{-1} w_t^2 + M'\lambda^2 w^2) \mathcal{C}^2 \dif x \leq N\int_{G_{c^2}} \bigl(\lambda |\nabla_{t,x} w|^2 + \lambda^3 w^2\bigr)\mathcal{C}^2 \dif x \dif t \\
        &+ N\int_{G_{c^2}}\bigl(|\Delta q|^2 + |\nabla q|^2 + q^2\bigr) \mathcal{C}^2 \dif x \dif t+ N\int_{\Sigma_T} \bigl(\lambda^2\widetilde{g}_t^2  + \widetilde{g}_{tt}^2 + \widetilde{h}_t^2\bigr) \mathcal{C}^2 \dif S .
    \end{aligned}
\end{equation}
Adding \eqref{1.60} and \eqref{1.59}, and adjusting the coefficients, we obtain \eqref{3.34}.

Similarly, for the $\bar{v}$-system, \eqref{1.35} gives $\bar{v}(0,\cdot)\equiv0$, so $\nabla\bar{v}(0,\cdot)\equiv0$. Therefore, by applying Proposition \ref{prop-2511}, every term $F^t|_{t=0}$ carries either an explicit factor $t$ or a factor $\bar{v}$ or $\nabla\bar{v}$, so the flux through $\partial_3G_{c^2}$ vanishes and Lemma \ref{lemma-3.3} applies with the same structure.
This gives \eqref{3.341}.
The proof is done.
\end{proof}

Based on this, we present the next lemma.
\begin{lem} \label{lemma-4.3}
    Under the condition of Theorem \ref{thm}, $w\in H^2(Q_T)$ is a solution to \eqref{w} and $\mathcal{C}$, $\beta $ are defined in \eqref{CWF}. There exists a positive constant $N$ such that the following inequality holds:
    \begin{align}
       & \, \lambda\int_{G_{c^2+2\sigma}} \bigl| \Delta(\nabla q \cdot \nabla f_2 ) \bigr|^2 \mathcal{C}^2 +\int_{G_{c^2+2\sigma}} \bigl(|\nabla_{t,x} w|^2 + \lambda^2 w^2\bigr) \mathcal{C}^2 \dif x \dif t \nonumber \\
        \leq & \, N \lambda^3 \exp\bigl[2\lambda(c^2 + 2\sigma)\bigr] \|w\|_{H^1(Q_T)}^2 + N\lambda^3 e^{2\lambda d^2}(\norm{w}^2_{H^{1}(\Sigma_T)}+\norm{\partial_\nu w}^2_{L^{2}(\Sigma_T)}) \nonumber \\
        & + N\lambda\int_{Q_\xi} \bigl(|\Delta q|^2 + |\nabla q|^2 + q^2+|\Delta \hat{f}|^2+|\nabla \hat{f}|^2+\hat{f}^2\bigr) \mathcal{C}^2 \dif x \dif t \nonumber \\
         & + N\lambda\int_{G_{c^2+2\sigma}} \bigl|\Delta (p_1^2\Delta \hat{f})\bigr|^2 \mathcal{C}^2 \dif x \dif t. \label{1.74}
	\end{align}
\end{lem}
\begin{proof}
The initial boundary term on the right-hand side of \eqref{3.22} can be treated using \eqref{3.34}, and apply Lemma \ref{lemma-3.1}, we obtain
\begin{equation}\label{1.67}
    \begin{aligned}
        \int_{G_{c^2+2\sigma}} &(\lambda |\nabla_{t,x} w|^2 + \lambda^3 w^2) \mathcal{C}^2 \dif x \dif \tau\leq N\int_{Q_\xi} \lambda\bigl(|\Delta q|^2 + |\nabla q|^2 + q^2\bigr) \mathcal{C}^2 \dif x \dif \tau\\
        &+N\lambda^2\int_{G_{c^2}\backslash G_{c^2+2\sigma}} (|\nabla_{t,x} w|^2 + \lambda^2w^2) \mathcal{C}^2 \dif x \dif \tau \\
		&+ N \exp\bigl[2\lambda(c^2 + 2\sigma)\bigr] \|w\|_{H^1(Q_T)}^2+ N\lambda^2\int_{\Sigma_T} (|\nabla_{t,x} w|^2 + \lambda^2 w^2) \mathcal{C}^2 \dif S \dif \tau.
    \end{aligned}
\end{equation}
By \eqref{3.4}, the second term on the right-hand side of \eqref{1.67} is bounded by $N\lambda^4\exp[2\lambda(c^2+2\sigma)]\norm{w}_{H^1(Q_T)}^2$. Hence, by further simplification, we obtain
\begin{align}
	\int_{G_{c^2+2\sigma}} &(|\nabla_{t,x} w|^2 + \lambda^2 w^2) \mathcal{C}^2 \dif x \dif \tau\leq N\int_{Q_\xi}  \bigl(|\Delta q|^2 + |\nabla q|^2 + q^2\bigr) \mathcal{C}^2 \dif x \dif \tau \nonumber \\
	& + N \lambda^3 \exp\bigl[2\lambda(c^2 + 2\sigma)\bigr] \|w\|_{H^1(Q_T)}^2+ N\int_{\Sigma_T} (\lambda|\nabla_{t,x} w|^2 + \lambda^3 w^2) \mathcal{C}^2 \dif S \dif \tau.\nonumber
\end{align}
In a local orthogonal coordinate system $(t,y',\nu)$,
\[|\nabla_{t,x} w|^2=|\partial_t w|^2+|\nabla_{y'} w|^2+|\partial_\nu w|^2.\]
Using the trace decomposition and $\mathcal{C}^2\leq e^{2\lambda d^2}$ on $\Sigma_T$, we obtain
\begin{equation}\label{1.68} 
	\begin{aligned}
	 &\int_{G_{c^2+2\sigma}} \bigl(|\nabla_{t,x} w|^2 + \lambda^2 w^2\bigr) \mathcal{C}^2 \dif x \dif t \leq  \, N \int_{Q_\xi} \bigl(|\Delta q|^2 + |\nabla q|^2 + q^2\bigr) \mathcal{C}^2 \dif x \dif t \\
     &+ N \lambda^3 \exp\bigl[2\lambda(c^2 + 2\sigma)\bigr] \|w\|_{H^1(Q_T)}^2 + N\lambda^3 e^{2\lambda d^2}(\norm{w}^2_{H^{1}(\Sigma_T)}+\norm{\partial_\nu w}^2_{L^{2}(\Sigma_T)}).
\end{aligned}
\end{equation}
By modifying the conditions of Lemma \ref{lemma-3.4} to \eqref{1.432}, we obtain 
    \begin{equation*}\label{gai1}
    \begin{aligned}
         \frac{\lambda\beta }{4a^4}\int_{G_{c^2+2\sigma}} \bigl| \Delta&(\nabla q \cdot \nabla f_2) \bigr|^2\mathcal{C}^2-\int_{G_{c^2+2\sigma}} w_t^2 \mathcal{C}^2\leq \lambda \beta \int_{G_{c^2+2\sigma}} v_t^2 \mathcal{C}^2 +\lambda \beta \int_{G_{c^2+2\sigma}} \bigl|\Delta (p_1^2\Delta \hat{f})\bigr|^2\mathcal{C}^2\\
		 &+ N\lambda \beta \int_{Q_\xi} ( \sum_{i,j=1}^n q_{ij}^2 + |\nabla q|^2 + q^2 ) \mathcal{C}^2+ N\lambda \beta \int_{Q_\xi} ( \sum_{i,j=1}^n \hat{f}_{ij}^2+|\nabla \hat{f}|^2+\hat{f}^2 ) \mathcal{C}^2.
    \end{aligned}
\end{equation*}
Adding it and \eqref{1.68}, the proof is complete.
\end{proof}

To obtain the final estimates for $f_1-f_2$ and $q$, we present the following important lemma.
\begin{lem}
    Under the conditions of Theorem \ref{thm} and the expansion \eqref{mathcalD}, $v$ satisfies \eqref{1.43} and $w$ satisfies \eqref{1.47}. $\sigma$ satisfies \eqref{sigma}, $d$ is defined in \eqref{d} and $\mathcal{C}$ is the Carleman weight function defined in \eqref{CWF}. There exists a positive constant $N$ such that
    \begin{align}
		\lambda e^{2\lambda(c^2 + 3\sigma)}& (\|\hat{f}\|_{H^1(\Omega)}^2 + \|q\|_{H^2(\Omega)}^2) + e^{2\lambda(c^2 + 3\sigma)} ( \|v\|_{H^1(G_{c^2+3\sigma})}^2 + \|w\|_{H^1(G_{c^2+3\sigma})}^2 ) \nonumber \\
		& \leq N\lambda^3  e^{2\lambda(c^2 + 2\sigma)} (\|v\|_{H^1(Q_T)}^2+\|w\|_{H^1(Q_T)}^2) \label{4.30}\\
		&+N\lambda^3 e^{2\lambda d^2} (\norm{v}^2_{H^{1}(\Sigma_T)}+\norm{\partial_\nu v}^2_{L^{2}(\Sigma_T)}+\norm{w}^2_{H^{1}(\Sigma_T)}+\norm{\partial_\nu w}^2_{L^{2}(\Sigma_T)}).\nonumber
	\end{align}
\end{lem}

\begin{proof}
	Under condition \eqref{eq:con1}, we can have $q |_{\partial\Omega} = \frac{\partial^k q}{\partial \nu^k} \Big|_{\partial\Omega} = 0$, for $k=1,2$.
	Using Lemma \ref{lemma-4} and the trace condition of $q$ on $\partial \Omega$, from the second-order matching of \eqref{eq:con1}, all second-order derivatives are zero on $\partial\Omega$. Since every term in $W$ contains $u, u_i, u_{ij}$, it follows that $\int_{\partial\Omega} W \cdot \nu \dif S = 0 .$ We can estimate $|\Delta(\nabla q \cdot \nabla f_2)|^2$ as follows,
	\begin{align*}
		N\lambda \int_{\Omega}(\sum_{i,j=1}^n q_{ij}^2  + \lambda^2|\nabla q|^2 + \lambda^4 q^2) \mathcal{C}^2 & \leq \int_{\Omega}[\Delta(\nabla q \cdot \nabla f_2)]^2 \mathcal{C}^2.
	\end{align*}
	Integrating these two inequalities with respect to the time variable in $(0,\xi)$ gives
	\begin{align}
		N\lambda \int_{Q_{\xi}}(\sum_{i,j=1}^n q_{ij}^2  + \lambda^2|\nabla q|^2 + \lambda^4 q^2)\mathcal{C}^2 & \leq \int_{Q_{\xi}}[\Delta(\nabla q \cdot \nabla f_2)]^2 \mathcal{C}^2. \label{gj3} 
	\end{align}
	Combining \eqref{3.341} and \eqref{1.73}, we obtain
	\begin{equation}\label{1v}
		\begin{aligned}
			& \lambda \int_{G_{c^2+2\sigma}} \bigl|\Delta (p_1^2\Delta \hat{f})\bigr|^2\mathcal{C}^2 \dif x \dif t + \lambda \int_{G_{c^2+2\sigma}} \bigl(|\nabla_{t,x} v|^2 + \lambda^2 v^2  \bigr) \mathcal{C}^2 \dif x \dif t\\
		 &\leq N\lambda^4  \exp\bigl[2\lambda(c^2 + 2\sigma)\bigr] \|v\|_{H^1(Q_T)}^2+ N\lambda^2\int_{\Sigma_T}(\bigl|\nabla_{x,t} v\bigr|^2+\lambda^2v^2)\mathcal{C}^2 \\
		& + N\int_{G_{c^2+2\sigma}}w_t^2\mathcal{C}^2+N \lambda\int_{Q_\xi} (|\Delta q|^2 + |\nabla q|^2 + q^2 + |\Delta \hat{f}|^2 + |\nabla \hat{f}|^2 + \hat{f}^2) \mathcal{C}^2 \dif x \dif t.
		\end{aligned}
	\end{equation}
	Adding \eqref{1.74} of Lemma \ref{lemma-4.3} and \eqref{1v} with appropriate scaling, we obtain
	\begin{equation}\label{1}
		\begin{aligned}
		& \lambda \int_{Q_\xi} (\bigl|\Delta (p_1^2\Delta \hat{f})\bigr|^2+ \bigl| \Delta(\nabla q \cdot  \nabla f_2) \bigr|^2) \mathcal{C}^2 \dif x \dif t \\
		&+ \lambda \int_{G_{c^2+2\sigma}} \bigl(|\nabla_{t,x} v|^2 + \lambda^2 v^2+ |\nabla_{t,x} w|^2 + \lambda^2 w^2 \bigr) \mathcal{C}^2 \dif x \dif t\\
		& \leq N\lambda^4  \exp\bigl[2\lambda(c^2 + 2\sigma)\bigr] (\|v\|_{H^1(Q_T)}^2+\|w\|_{H^1(Q_T)}^2)\\
		&+ N\lambda^4 e^{2\lambda d^2}(\norm{v}^2_{H^{1}(\Sigma_T)} +\norm{\partial_\nu v}^2_{L^{2}(\Sigma_T)} +\norm{w}^2_{H^{1}(\Sigma_T)}+\norm{\partial_\nu w}^2_{L^{2}(\Sigma_T)})\\
		& + N \lambda\int_{Q_\xi} (|\Delta q|^2 + |\nabla q|^2 + q^2 + |\Delta \hat{f}|^2 + |\nabla \hat{f}|^2 + \hat{f}^2) \mathcal{C}^2 \dif x \dif t, 
	\end{aligned}
	\end{equation}
	where we have used the relation between $Q_{\xi}$ and $G_{c^2+2\sigma}$. Apply Lemma \ref{lemma-5} and condition \eqref{eq:con2} to $\int_{Q_\xi} \bigl|\Delta (p_1^2\Delta \hat{f})\bigr|^2\mathcal{C}^2$ yields
	\begin{equation}\label{guji1}
			\lambda \int_{Q_\xi} \bigl|\Delta (p_1^2\Delta \hat{f})\bigr|^2\mathcal{C}^2 \dif x \dif t \geq N\lambda^2 \int_{Q_\xi} \bigl|\nabla(p_1^2\Delta \hat{f})\bigr|^2\mathcal{C}^2\dif x \dif t+N\lambda^4 \int_{Q_\xi} \bigl|p_1^2\Delta \hat{f}\bigr|^2 \mathcal{C}^2 \dif x \dif t.
	\end{equation}
	This step uses the boundary matching up to order three of \eqref{eq:con2}, $\hat{f}|_{\partial \Omega}=\partial_\nu \hat{f}|_{\partial \Omega}=\partial_\nu^2 \hat{f}|_{\partial \Omega}=\partial_\nu^3 \hat{f}|_{\partial \Omega}=0.$ In this case, applying \eqref{guji1} to the left-hand side of \eqref{1} gives
	\begin{align}
		& \lambda^2 \int_{Q_\xi} \bigl|\nabla(p_1^2\Delta \hat{f})\bigr|^2\mathcal{C}^2\dif x \dif t+\lambda^4 \int_{Q_\xi} \bigl|p_1^2\Delta \hat{f}\bigr|^2 \mathcal{C}^2 \dif x \dif t+\lambda\int_{Q_\xi} \bigl| \Delta(\nabla q \cdot  \nabla f_2) \bigr|^2 \mathcal{C}^2 \dif x \dif t\nonumber \\
		&+ \lambda \int_{G_{c^2+2\sigma}} \bigl(|\nabla_{t,x} v|^2 + \lambda^2 v^2 + |\nabla_{t,x} w|^2 + \lambda^2 w^2 \bigr) \mathcal{C}^2 \dif x \dif t \nonumber \\
		&\leq N\lambda^4  \exp\bigl[2\lambda(c^2 + 2\sigma)\bigr] (\|v\|_{H^1(Q_T)}^2+\|w\|_{H^1(Q_T)}^2)\label{2} \\
		&+ N\lambda^4 e^{2\lambda d^2}(\norm{w}^2_{H^{1}(\Sigma_T)}+\norm{\partial_\nu w}^2_{L^{2}(\Sigma_T)}+\norm{v}^2_{H^{1}(\Sigma_T)}+\norm{\partial_\nu v}^2_{L^{2}(\Sigma_T)}) \nonumber \\
		& + N \lambda\int_{Q_\xi} (|\Delta q|^2 + |\nabla q|^2 + q^2 + |\Delta \hat{f}|^2 + |\nabla \hat{f}|^2 + \hat{f}^2) \mathcal{C}^2 \dif x \dif t. \nonumber
	\end{align}
	We use Lemma \ref{lemma-5} to analyze the first term on the left-hand side of \eqref{2}.
	\begin{equation*}
		\begin{aligned}
			\lambda^2 \int_{Q_\xi} \bigl|\nabla(p_1^2\Delta \hat{f})\bigr|^2\mathcal{C}^2\dif x \dif t&=\lambda^2 \int_{Q_\xi}(2p_1\nabla p_1\Delta \hat{f}+p_1^2\nabla \Delta \hat{f})^2\mathcal{C}^2\dif x \dif t\\
          			&\geq N\lambda^2\int_{Q_\xi} |\nabla \Delta \hat{f}|^2\mathcal{C}^2\dif x \dif t-\lambda^2 \int_{Q_\xi}|\nabla p_1|^2|\Delta \hat{f}|^2\mathcal{C}^2\dif x \dif t\\
					&\geq N\lambda^2\int_{Q_\xi} |\Delta \nabla  \hat{f}|^2\mathcal{C}^2\dif x \dif t-N\lambda^2\int_{Q_\xi}|\Delta \hat{f}|^2\mathcal{C}^2\dif x \dif t.
		\end{aligned}
	\end{equation*}
	Using this inequality, \eqref{2} can be written as
    \begin{align}
		& N\lambda^2 \int_{Q_\xi} (|\nabla \Delta \hat{f}|^2+\lambda^2|\Delta \hat{f}|^2 )\mathcal{C}^2 \dif x \dif t +\lambda\int_{Q_\xi} \bigl| \Delta(\nabla q \cdot  \nabla f_2) \bigr|^2 \mathcal{C}^2 \dif x \dif t\nonumber \\
		& + C\lambda \int_{G_{c^2+2\sigma}} \bigl(|\nabla_{t,x} v|^2 + \lambda^2 v^2 + |\nabla_{t,x} w|^2 + \lambda^2 w^2 \bigr) \mathcal{C}^2 \dif x \dif t \nonumber \\
		&\leq N\lambda^4  \exp\bigl[2\lambda(c^2 + 2\sigma)\bigr] (\|v\|_{H^1(Q_T)}^2+\|w\|_{H^1(Q_T)}^2) \nonumber\\
		&+ N\lambda^4 e^{2\lambda d^2}(\norm{w}^2_{H^{1}(\Sigma_T)}+\norm{\partial_\nu w}^2_{L^{2}(\Sigma_T)}+\norm{v}^2_{H^{1}(\Sigma_T)}+\norm{\partial_\nu v}^2_{L^{2}(\Sigma_T)}) \nonumber \\
		& + N \lambda\int_{Q_\xi} (|\Delta q|^2 + |\nabla q|^2 + q^2 + |\Delta \hat{f}|^2 + |\nabla \hat{f}|^2 + \hat{f}^2) \mathcal{C}^2 \dif x \dif t,
	\end{align}
	Combining this with \eqref{gj3}, we have
	\begin{align}
		& N\lambda^2 \int_{Q_\xi} (|\nabla \Delta \hat{f}|^2+\lambda^2|\Delta \hat{f}|^2 +\sum_{i,j=1}^n q_{ij}^2+\lambda^2|\nabla q|^2 + \lambda^4 q^2)\mathcal{C}^2 \dif x \dif t\nonumber \\
		& + C\lambda \int_{G_{c^2+2\sigma}} \bigl(|\nabla_{t,x} v|^2 + \lambda^2 v^2 + |\nabla_{t,x} w|^2 + \lambda^2 w^2 \bigr) \mathcal{C}^2 \dif x \dif t \nonumber \\
		&\leq  N\lambda^4  \exp\bigl[2\lambda(c^2 + 2\sigma)\bigr] (\|v\|_{H^1(Q_T)}^2+\|w\|_{H^1(Q_T)}^2) \label{1.77}\\
		&+ N\lambda^4 e^{2\lambda d^2}(\norm{w}^2_{H^{1}(\Sigma_T)}+\norm{\partial_\nu w}^2_{L^{2}(\Sigma_T)}+\norm{v}^2_{H^{1}(\Sigma_T)} +\norm{\partial_\nu v}^2_{L^{2}(\Sigma_T)}) \nonumber \\
		& + N \lambda\int_{Q_\xi} (|\Delta q|^2 + |\nabla q|^2 + q^2 + |\Delta \hat{f}|^2 + |\nabla \hat{f}|^2 + \hat{f}^2) \mathcal{C}^2 \dif x \dif t.\nonumber
	\end{align}
	We continue by applying Lemma \ref{lemma-5} to each $\partial_i \hat{f}$ component and summing
	\begin{equation*}
		\begin{aligned}
			N\lambda^2 \int_{Q_\xi} (|\nabla \Delta  \hat{f}|^2+\lambda^2|\Delta \hat{f}|^2)\mathcal{C}^2\dif x \dif t&= N\lambda^2\int_{Q_\xi} (|\Delta \nabla  \hat{f}|^2+\lambda^2|\Delta \hat{f}|^2)\mathcal{C}^2\dif x \dif t\\
			&\geq N\lambda^3\int_{Q_\xi} (|\Delta \hat{f}|^2+\lambda^2|\nabla \hat{f}|^2)\mathcal{C}^2\dif x \dif t.
		\end{aligned}
	\end{equation*}
	Using this inequality, \eqref{1.77} can be written as
    \begin{align}
		& N\lambda^2 \int_{Q_\xi} (\lambda|\Delta \hat{f}|^2 +\lambda^3|\nabla \hat{f}|^2+\sum_{i,j=1}^n q_{ij}^2+\lambda^2|\nabla q|^2 + \lambda^4 q^2) \mathcal{C}^2 \dif x \dif t \nonumber \\
		& + C\lambda \int_{G_{c^2+2\sigma}} \bigl(|\nabla_{t,x} v|^2 + \lambda^2 v^2 + |\nabla_{t,x} w|^2 + \lambda^2 w^2 \bigr) \mathcal{C}^2 \dif x \dif t \nonumber \\
		&\leq   N\lambda^4  \exp\bigl[2\lambda(c^2 + 2\sigma)\bigr] (\|v\|_{H^1(Q_T)}^2+\|w\|_{H^1(Q_T)}^2) \label{1.770} \\
		&+ N\lambda^4 e^{2\lambda d^2}(\norm{w}^2_{H^{1}(\Sigma_T)}+\norm{\partial_\nu w}^2_{L^{2}(\Sigma_T)}+\norm{v}^2_{H^{1}(\Sigma_T)} +\norm{\partial_\nu v}^2_{L^{2}(\Sigma_T)}) \nonumber \\
		& + N \lambda\int_{Q_\xi} (|\Delta q|^2 + |\nabla q|^2 + q^2 + |\Delta \hat{f}|^2 + |\nabla \hat{f}|^2 + \hat{f}^2) \mathcal{C}^2 \dif x \dif t. \nonumber
	\end{align}
	In \eqref{1.770}, for $\lambda \geq \lambda_1(\Omega, x_0, \beta, N)$, the last term (the integral on $Q_\xi$) on the right-hand side can be absorbed by the left-hand side, thus we obtain 
	\begin{align}
		& N\lambda^2 \int_{Q_\xi} (\lambda|\Delta \hat{f}|^2 +\lambda^3|\nabla \hat{f}|^2+\sum_{i,j=1}^n q_{ij}^2+\lambda^2|\nabla q|^2 + \lambda^4 q^2) \mathcal{C}^2 \dif x \dif t \nonumber \\
		& + C\lambda \int_{G_{c^2+2\sigma}} \bigl(|\nabla_{t,x} v|^2 + \lambda^2 v^2 + |\nabla_{t,x} w|^2 + \lambda^2 w^2 \bigr) \mathcal{C}^2 \dif x \dif t\nonumber\\
		& \leq N \lambda\int_{Q_\xi} \hat{f}^2 \mathcal{C}^2 \dif x \dif t +N\lambda^4  \exp\bigl[2\lambda(c^2 + 2\sigma)\bigr] (\|v\|_{H^1(Q_T)}^2+\|w\|_{H^1(Q_T)}^2)\nonumber\\
		&+ N\lambda^4 e^{2\lambda d^2}(\norm{w}^2_{H^{1}(\Sigma_T)}+\norm{\partial_\nu w}^2_{L^{2}(\Sigma_T)}+\norm{v}^2_{H^{1}(\Sigma_T)}+\norm{\partial_\nu v}^2_{L^{2}(\Sigma_T)}). \nonumber
	\end{align}
	Applying Lemma \ref{lemma-5} again, we have
    \begin{align}
		& N\lambda^2 \int_{Q_\xi} (\lambda^3|\nabla \hat{f}|^2+\lambda^4 \hat{f}^2+\sum_{i,j=1}^n q_{ij}^2+\lambda^2|\nabla q|^2 + \lambda^4 q^2) \mathcal{C}^2 \dif x \dif t \nonumber \\
		& + C\lambda \int_{G_{c^2+2\sigma}} \bigl(|\nabla_{t,x} v|^2 + \lambda^2 v^2 + |\nabla_{t,x} w|^2 + \lambda^2 w^2 \bigr) \mathcal{C}^2 \dif x \dif t \label{1.771} \\
		&\leq N\lambda^4  \exp\bigl[2\lambda(c^2 + 2\sigma)\bigr] (\|v\|_{H^1(Q_T)}^2+\|w\|_{H^1(Q_T)}^2) \nonumber\\
		&+ N\lambda^4 e^{2\lambda d^2}(\norm{w}^2_{H^{1}(\Sigma_T)}+\norm{\partial_\nu w}^2_{L^{2}(\Sigma_T)}+\norm{v}^2_{H^{1}(\Sigma_T)}+\norm{\partial_\nu v}^2_{L^{2}(\Sigma_T)}). \nonumber
	\end{align}
	Since $Q_\xi \subset G_{c^2+3\sigma} \subset G_{c^2+2\sigma}$ and $\mathcal{C}^2(t,x) \geq e^{2\lambda(c^2 + 3\sigma)}$ in $G_{c^2+3\sigma}$, \eqref{1.771} implies that
	\begin{align*}
		& \lambda^2 e^{2\lambda(c^2 + 3\sigma)} (\|\hat{f}\|_{H^1(\Omega)}^2 + \|q\|_{H^2(\Omega)}^2) + \lambda e^{2\lambda(c^2 + 3\sigma)} ( \|v\|_{H^1(G_{c^2+3\sigma})}^2 + \|w\|_{H^1(G_{c^2+3\sigma})}^2 ) \nonumber \\
		&\leq  N\lambda^4  e^{2\lambda(c^2 + 2\sigma)} (\|v\|_{H^1(Q_T)}^2+\|w\|_{H^1(Q_T)}^2)\nonumber\\
		&+N\lambda^4 e^{2\lambda d^2} (\norm{w}^2_{H^{1}(\Sigma_T)}+\norm{\partial_\nu w}^2_{L^{2}(\Sigma_T)}+\norm{v}^2_{H^{1}(\Sigma_T)}+\norm{\partial_\nu v}^2_{L^{2}(\Sigma_T)}).
	\end{align*}
	The proof is done.
\end{proof}

We handle $\hat{f}$ and $q$ separately.
Dividing \eqref{4.30} by $\lambda e^{2\lambda(c^2 + 3\sigma)}$ and ignoring $\|v\|_{H^1( G_{c^2+3\sigma})}^2$ and $\|w\|_{H^1 (G_{c^2+3\sigma})}^2$, we obtain
\begin{align}
		\|&\hat{f}\|_{H^1(\Omega)}^2 + \|q\|_{H^2(\Omega)}^2 \leq N \lambda^2 e^{-2\lambda\sigma} (\|v\|_{H^1(Q_T)}^2+\|w\|_{H^1(Q_T)}^2) \nonumber \\
		& +N \lambda^2 e^{2\lambda(d^2-c^2 - 3\sigma)} (\norm{v}^2_{H^{1}(\Sigma_T)}+\norm{\partial_\nu v}^2_{L^{2}(\Sigma_T)}+\norm{w}^2_{H^{1}(\Sigma_T)}+\norm{\partial_\nu w}^2_{L^{2}(\Sigma_T)}). \label{1.82}
\end{align}
Similarly, dividing \eqref{4.30} by $e^{2\lambda(c^2 + 3\sigma)}$ and ignoring $\|\hat{f}\|_{H^1(\Omega)}^2$ and $\|q\|_{H^2(\Omega)}^2$ gives
\begin{align}
		\|&v\|_{H^1(G_{c^2+3\sigma})}^2 +  \|w\|_{H^1(G_{c^2+3\sigma})}^2 \leq N\lambda^3  e^{-2\lambda\sigma} (\|v\|_{H^1(Q_T)}^2+\|w\|_{H^1(Q_T)}^2) \nonumber \\
		& +N \lambda^3 e^{2\lambda(d^2-c^2 - 3\sigma)} (\norm{v}^2_{H^{1}(\Sigma_T)}+\norm{\partial_\nu v}^2_{L^{2}(\Sigma_T)}+\norm{w}^2_{H^{1}(\Sigma_T)}+\norm{\partial_\nu w}^2_{L^{2}(\Sigma_T)}). \label{1.81}
\end{align}
We can replace the norms of the space $H^1(G_{c^2+3\sigma})$ on the left hand side of \eqref{1.81} with the norms in the space $H^1(Q_\xi)$. We obtain
\begin{align}
	\|v\|_{H^1(Q_\xi)}^2 &+ \|w\|_{H^1(Q_\xi)}^2\leq N\lambda^3  e^{-2\lambda\sigma} (\|v\|_{H^1(Q_T)}^2+\|w\|_{H^1(Q_T)}^2)  \label{1.83}\\
	& +N \lambda^3 e^{2\lambda(d^2-c^2 - 3\sigma)} (\norm{v}^2_{H^{1}(\Sigma_T)}+\norm{\partial_\nu v}^2_{L^{2}(\Sigma_T)}+\norm{w}^2_{H^{1}(\Sigma_T)}+\norm{\partial_\nu w}^2_{L^{2}(\Sigma_T)}). \nonumber
\end{align}
On the other hand, there exists a number $t_1 \in (0,\xi)$ such that
\begin{multline*}
\|v(\cdot,t_1)\|_{H^1(\Omega)}^2 + \|v_t(\cdot,t_1)\|_{L^2(\Omega)}^2 + \|w(\cdot,t_1)\|_{H^1(\Omega)}^2 + \|w_t(\cdot,t_1)\|_{L^2(\Omega)}^2 \\
\leq \frac{1}{\xi} ( \|v\|_{H^1(Q_\xi)}^2 + \|w\|_{H^1(Q_\xi)}^2 ).
\end{multline*}
Hence \eqref{1.83} implies that
\begin{align}
	\|&v(\cdot,t_1)\|_{H^1(\Omega)}^2 + \|v_t(\cdot,t_1)\|_{L^2(\Omega)}^2 \nonumber \\
		&\leq N\frac{\lambda^3}{\xi}  e^{-2\lambda\sigma} (\|v\|_{H^1(Q_T)}^2+\|w\|_{H^1(Q_T)}^2) \nonumber \\
	& +N \frac{\lambda^3}{\xi} e^{2\lambda(d^2-c^2 - 3\sigma)} (\norm{v}^2_{H^{1}(\Sigma_T)}+\norm{\partial_\nu v}^2_{L^{2}(\Sigma_T)}+\norm{w}^2_{H^{1}(\Sigma_T)}+\norm{\partial_\nu w}^2_{L^{2}(\Sigma_T)}), \label{1.85}
\end{align}
and
\begin{align}
	\|&w(\cdot,t_1)\|_{H^1(\Omega)}^2 + \|w_t(\cdot,t_1)\|_{L^2(\Omega)}^2\leq N\frac{\lambda^3}{\xi}  e^{-2\lambda\sigma} (\|v\|_{H^1(Q_T)}^2+\|w\|_{H^1(Q_T)}^2) \nonumber \\
	& +N \frac{\lambda^3}{\xi} e^{2\lambda(d^2-c^2 - 3\sigma)} (\norm{v}^2_{H^{1}(\Sigma_T)}+\norm{\partial_\nu v}^2_{L^{2}(\Sigma_T)}+\norm{w}^2_{H^{1}(\Sigma_T)}+\norm{\partial_\nu w}^2_{L^{2}(\Sigma_T)}). \nonumber
\end{align}
 
The estimate obtained so far is localized in a very flat region near $t=0$. To obtain our target estimate, the following treatment is required.
Consider now the hyperbolic equation \eqref{1.4} for the function $v(t,x)$ in the cylinder $(t_1, T) \times \Omega$ with the boundary conditions
\[
v |_{(t_1,T) \times \partial\Omega} = \widetilde{g}(t,x),\quad \partial_\nu v|_{(t_1,T)\times\partial\Omega}=\widetilde h(t,x) 
\]
and with the initial conditions
\[
v |_{t=t_1} = v(t_1,x),\quad v_t|_{t=t_1}=v_t(t_1,x).
\]
Then, by performing an energy estimate and applying Lemma \ref{lemma-3} and \eqref{1.85}, we obtain
\begin{align}
\|v\|_{H^1((t_1,T)\times\Omega)}^2&\leq N\frac{\lambda^3}{\xi}  e^{-2\lambda\sigma} (\|v\|_{H^1(Q_T)}^2+\|w\|_{H^1(Q_T)}^2) \nonumber \\
&+N \frac{\lambda^3}{\xi} e^{2\lambda(d^2-c^2 - 3\sigma)} (\norm{v}^2_{H^{1}(\Sigma_T)}+\norm{\partial_\nu v}^2_{L^{2}(\Sigma_T)}+\norm{w}^2_{H^{1}(\Sigma_T)}+\norm{\partial_\nu w}^2_{L^{2}(\Sigma_T)}) \nonumber\\
&+ N\norm{\mathcal{A}(p_1,p_2)}^2_{L^2((t_1,T)\times\Omega)}. \label{1.87}
\end{align}

Make a change of variables $(t,x) \Leftrightarrow (\tau = t_1 - t,x)$. Then we obtain the same hyperbolic equation \eqref{1.4}, in which $v(\tau,x) := v(t_1 - \tau,x)$ and $t$ is replaced with $t_1 - \tau$ in all coefficients. This reflects the fact that the hyperbolic equation can be solved in both the positive and negative directions of time. This new hyperbolic equation is satisfied in the cylinder $(0,t_1) \times \Omega$ with the boundary conditions
\[
v|_{(0,t_1) \times \partial \Omega} = \widetilde{g}(t_1 - \tau, x), \quad \frac{\partial v}{\partial \nu} \Big|_{(0,t_1) \times \partial \Omega} = \widetilde{h}(t_1 - \tau, x), 
\]
and with the initial conditions
\[
v|_{\tau=0} = v(t_1, x), \quad v_\tau |_{\tau=0} = -v_t(t_1, x).
\]

Therefore using again Lemma \ref{lemma-3} and \eqref{1.85}, we obtain
\begin{align}
\|v\|_{H^1((0,t_1)\times\Omega)}^2 &\leq N\frac{\lambda^3}{\xi}  e^{-2\lambda\sigma} (\|v\|_{H^1(Q_T)}^2+\|w\|_{H^1(Q_T)}^2) \nonumber \\
& +N \frac{\lambda^3}{\xi} e^{2\lambda(d^2-c^2 - 3\sigma)} (\norm{v}^2_{H^{1}(\Sigma_T)}+\norm{\partial_\nu v}^2_{L^{2}(\Sigma_T)}+\norm{w}^2_{H^{1}(\Sigma_T)}+\norm{\partial_\nu w}^2_{L^{2}(\Sigma_T)}) \nonumber\\
& + N\norm{\mathcal{A}(p_1,p_2)}^2_{L^2((0,t_1)\times\Omega)}.
\end{align}

Summing up this estimate with \eqref{1.87} we obtain
\begin{align}
	\|v\|_{H^1(Q_T)}^2 \leq &N\frac{\lambda^3}{\xi}  e^{-2\lambda\sigma} (\|v\|_{H^1(Q_T)}^2+\|w\|_{H^1(Q_T)}^2) \nonumber \\
	& +N \frac{\lambda^3}{\xi} e^{2\lambda(d^2-c^2 - 3\sigma)} (\norm{v}^2_{H^{1}(\Sigma_T)}+\norm{\partial_\nu v}^2_{L^{2}(\Sigma_T)}+\norm{w}^2_{H^{1}(\Sigma_T)}+\norm{\partial_\nu w}^2_{L^{2}(\Sigma_T)})\nonumber\\
	& + N\norm{\mathcal{A}(p_1,p_2)}^2_{L^2(Q_T)},
\end{align}
similarly,
\begin{align}
	\|w\|_{H^1(Q_T)}^2 & \leq N\frac{\lambda^3}{\xi}  e^{-2\lambda\sigma} (\|v\|_{H^1(Q_T)}^2+\|w\|_{H^1(Q_T)}^2) \nonumber \\
	& +N \frac{\lambda^3}{\xi} e^{2\lambda(d^2-c^2 - 3\sigma)} (\norm{v}^2_{H^{1}(\Sigma_T)}+\norm{\partial_\nu v}^2_{L^{2}(\Sigma_T)}+\norm{w}^2_{H^{1}(\Sigma_T)}+\norm{\partial_\nu w}^2_{L^{2}(\Sigma_T)}) \nonumber\\
	& + N\norm{\partial_t\mathcal{A}(p_1,p_2)}^2_{L^2(Q_T)}.
\end{align}

Summing up two latter estimates, we obtain
\begin{align}
	&(1 - \frac{N\lambda^3}{\xi e^{2\lambda\sigma}}) ( \|v\|_{H^1(Q_T)}^2 + \|w\|_{H^1(Q_T)}^2 )\leq  N\bigl(\norm{\mathcal{A}(p_1,p_2)}^2_{L^2(Q_T)}+\norm{\partial_t\mathcal{A}(p_1,p_2)}^2_{L^2(Q_T)}\bigr)\nonumber\\
	&+N \frac{\lambda^3}{\xi} e^{2\lambda (d^2-c^2-3\sigma)} (\norm{v}^2_{H^{1}(\Sigma_T)}+\norm{\partial_\nu v}^2_{L^{2}(\Sigma_T)}+\norm{w}^2_{H^{1}(\Sigma_T)}+\norm{\partial_\nu w}^2_{L^{2}(\Sigma_T)}).\label{1.91}
\end{align}

Choose a sufficiently large $\lambda_1 = \lambda_1(\Omega, x_0, T, N, \sigma, \xi)$ such that
\[
1 - \frac{N\lambda^3}{\xi e^{2\lambda\sigma}} \geq \frac{1}{2},
\]
hence \eqref{1.91} implies that
\begin{equation} \label{1.92}
	\begin{aligned}
		&\|v\|_{H^1(Q_T)}^2 + \|w\|_{H^1(Q_T)}^2
	\leq N\bigl(\norm{\mathcal{A}(p_1,p_2)}^2_{L^2(Q_T)}+\norm{\partial_t\mathcal{A}(p_1,p_2)}^2_{L^2(Q_T)}\bigr)\\
	&+  N \frac{\lambda^3}{\xi} e^{2\lambda (d^2-c^2-3\sigma)} (\norm{v}^2_{H^{1}(\Sigma_T)}+\norm{\partial_\nu v}^2_{L^{2}(\Sigma_T)}+\norm{w}^2_{H^{1}(\Sigma_T)}+\norm{\partial_\nu w}^2_{L^{2}(\Sigma_T)}).
	\end{aligned}
\end{equation}
Estimating the right-hand side of \eqref{1.82} by \eqref{1.92}, we obtain
\begin{equation*}
	\begin{aligned}
		\norm{\hat{f}}_{H^1(\Omega)}^2&+ \norm{q}_{H^2(\Omega)}^2
	\leq N \lambda^5 e^{2\lambda (d^2-c^2-3\sigma)} (\norm{v}^2_{H^{1}(\Sigma_T)}+\norm{\partial_\nu v}^2_{L^{2}(\Sigma_T)}+\norm{w}^2_{H^{1}(\Sigma_T)}
	\\ &+\norm{\partial_\nu w}^2_{L^{2}(\Sigma_T)})+ N \lambda^2 e^{-2\lambda\sigma}\bigl(\norm{\mathcal{A}(p_1,p_2)}^2_{L^2(Q_T)}+\norm{\partial_t\mathcal{A}(p_1,p_2)}^2_{L^2(Q_T)}\bigr).
	\end{aligned}
\end{equation*}
By \eqref{mathcalA} and the boundedness of the coefficients $h^{(2)}$, $h^{i}$, $h^{(0)}$ and their time derivatives, the last term is bounded by $N\lambda^2 e^{-2\lambda\sigma}\norm{q}^2_{H^2(\Omega)}$. Taking $\lambda$ large enough so that $N\lambda^2 e^{-2\lambda\sigma}\le\frac{1}{2}$, it is absorbed into the left-hand side, and we obtain
\begin{equation*}
	\begin{aligned}
		\norm{\hat{f}}_{H^1(\Omega)}^2 + \norm{q}_{H^2(\Omega)}^2\leq &N \lambda^5 e^{2\lambda (d^2-c^2-3\sigma)} (\norm{v}^2_{H^{1}(\Sigma_T)}+\norm{\partial_\nu v}^2_{L^{2}(\Sigma_T)}\\
		&+\norm{w}^2_{H^{1}(\Sigma_T)}+\norm{\partial_\nu w}^2_{L^{2}(\Sigma_T)}).
	\end{aligned}
\end{equation*}
Finally, on $\Sigma_T$ we have $g_i=u_i|_{\Sigma_T}$,
$v=\tilde g=p_1\partial_t^3(g_1-g_2)$,
$w=\tilde g_t=p_1\partial_t^4(g_1-g_2)$, and
$\partial_\nu v=\tilde h=p_1\partial_t^3(h_1-h_2)+\frac{\partial p_1}{\partial\nu}\partial_t^3(g_1-g_2)$,
$\partial_\nu w=\tilde h_t=p_1\partial_t^4(h_1-h_2)+\frac{\partial p_1}{\partial\nu}\partial_t^4(g_1-g_2)$.
From $\partial_tu(0,\cdot)=0$, we obtain $\widetilde{g}(0,\cdot)=\widetilde{h}(0,\cdot)=0$, which allows us to control the norm of the original function using the norm of its time derivative. Hence, using \eqref{g} and the trace theorem,
\[
\norm{v}_{H^1(\Sigma_T)} \le N\norm{\tilde g_t}_{H^1(\Sigma_T)}\le N\norm{\partial_t^4(g_1-g_2)}_{H^1(\Sigma_T)},
\]
\[
\norm{w}_{H^{1}(\Sigma_T)}\le N\norm{\partial_t^4(g_1-g_2)}_{H^{1}(\Sigma_T)},
\]
and
\[
\norm{\partial_\nu v}_{L^{2}(\Sigma_T)}=\norm{\widetilde{h}}_{L^{2}(\Sigma_T)}\leq N\norm{\widetilde{h}_t}_{L^{2}(\Sigma_T)},
\]
\[
\norm{\partial_\nu w}_{L^{2}(\Sigma_T)}\le N\bigl(\norm{\partial_t^4(h_1-h_2)}_{L^{2}(\Sigma_T)}+\norm{\partial_t^4(g_1-g_2)}_{L^{2}(\Sigma_T)}\bigr).
\]
Substituting these bounds into the previous inequality, the theorem is proved.

\section*{Acknowledgements}

The research of S.M.~is partially supported by the NSFC (No.~12301540).

\vspace{.3cm}
\noindent \textbf{Data Availability:} No data were generated or analysed in this study.

\vspace{.3cm}
\noindent \textbf{Conflict of Interest:} The authors declare that there is no conflict of interest regarding the publication of this paper.


{


}

\end{document}